\documentclass[twoside, 11pt]{article}
\usepackage{amsmath,amssymb,amsthm,mathrsfs}
\usepackage[margin=2.5cm]{geometry} % showframe,
\usepackage{lipsum}
\usepackage{titlesec,hyperref}
\usepackage{fancyhdr}
\usepackage{graphicx}
\usepackage[numbers,sort&compress]{natbib}
\usepackage{color}

\fancypagestyle{plain}
{
\fancyhf{}

}

\titleformat{\subsection}{\it}{\thesubsection.\enspace}{1.5pt}{}
\titleformat{\subsubsection}{\it}{\thesubsubsection.\enspace}{1.5pt}{}

\newtheorem{theo}{Theorem}[section]
\newtheorem{lemm}[theo]{Lemma}

\newtheorem{nota}[theo]{Notation}
\newtheorem{rema}{Remark}[section]
\numberwithin{equation}{section}

\allowdisplaybreaks

\def\th2{\frac{\theta}{2}}

\begin{document}
\title{The stability of current vortex sheets with transverse magnetic field\hspace{-4mm}}
\author{ Binqiang Xie $^{*}$,  Yueyang Feng  $^{\dag}$, Ying Zhang $^{\ddag}$ \\[10pt]
	\small {$^{*}$ School of Mathematics and Statistics,}\\
	\small { Guangdong University of Technology, Guangzhou, Guangdong,  510006, China}
	\\
       \small {$^\dag$ Institute of Applied Physics and Computational Mathematics, Beijing, 100868,  China}
	\\
	\small {$^\ddag$  National Center for Mathematics and Interdisciplinary Sciences,}\\
 \small { Academy of Mathematics and Systems Science,  Chinese Academy of Sciences, Beijing, 100190, China}\\[5pt]
}

%\address{Department of Mathematics and Statistics, Indian Institute of Technology Kanpur, \\Kanpur, Uttar Pradesh, India\\}

\footnotetext{
	
	E-mail addresses: {\it  xbqmath@gdut.edu.cn (B.Q. Xie).
    
    \quad \quad \quad \quad \quad \quad \quad  \quad fengyueyang23@gscaep.ac.cn (Y.Y. Feng).
    
    \quad \quad \quad \quad \quad \quad \quad  \quad   zhangying24@amss.ac.cn (Y. Zhang)}

    Corresponding Author: \it Y. Zhang (zhangying24@amss.ac.cn )
    
    }
\date{}

\maketitle
\begin{abstract}
Compared to the results in \cite{Shivamoggi}, using the normal mode method, we have rigorously confirmed that a transverse magnetic field reduces the stability of the system. Specifically, a larger velocity is required for stability in the presence of a magnetic field than in its absence. More precisely, when the magnitude of the magneto-acoustic Mach number $M_{B}:=\frac{\dot{v}_1^{+}}{\bar{C}_{B}}>\sqrt{2}$, we proved the  well-posedness of the current  vortex sheet problem for compressible MHD flows with a transverse magnetic field. 

\vspace*{5pt}

\vspace*{5pt}
\end{abstract}

%\tableofcontents

\section{Introduction}
\quad
The vortex sheet is an idealized model of shear flow, typically representing a surface where there is a discontinuity in vorticity within the fluid. In other words, the vortex sheet is a surface where the velocity experiences a jump, resulting in a discontinuity in the tangential velocity of the fluid. This phenomenon typically arises between two distinct fluid regions, such as different flow directions. The study of vortex sheets not only deepens the understanding of fundamental problems in fluid mechanics but also advances the interdisciplinary development of a variety of spaces and astrophysical phenomena, involving sheared plasma flow such as the stability of the interface between the solar wind and the magnetosphere \cite{Dungey,Parker}, interactions between adjacent streams of different velocities in the solar wind \cite{Sturrock} and the dynamic structure of cometary tails \cite{Ershkovich}.  In addition to these visible wavy or vortex structures, large-scale phenomena are also observed by spacecraft in the velocity-shear regions of space plasma both at the flank magnetopause \cite{Ch}.
% It enables a deeper understanding of practical magnetohydrodynamic (MHD) problems, such as plasma boundary layers, stability issues, and shear flow in the presence of magnetic fields.

In the context of 2D Euler equations in gas dynamics, rectilinear vortex sheets are linearly stable when the Mach number $M>\sqrt{2}$, and they  become violently unstable when the Mach number $M<\sqrt{2}$ \cite{Miles,Miles1957}. In pioneering works, Coulombel-Secchi \cite{Coulombel, Secchi} demonstrated the nonlinear stability and local existence of vortex sheets for ideal compressible fluids using micro-local analysis and the Nash-Moser method. Subsequently, Morando-Trebeschi-Wang \cite{Trebeschi1, Trebeschi2} extended this result to 2D ideal nonisentropic compressible flows. Recently, numerous stability results for 2D systems have been obtained. The method developed by Coulombel-Secchi \cite{Coulombel} has been applied to 2D magnetohydrodynamics (MHD) flows, leading to the derivation of a necessary and sufficient condition for the linear stability of rectilinear current vortex sheets by Wang-Yu \cite{Yu1}. For vortex sheets in elastic flows, Chen-Hu-Wang and Chen-Hu-Wang-Wang-Yuan \cite{Hu, Chen, Chen3} have established the nonlinear stability and local existence of compressible vortex sheets for 2D isentropic elastic fluids.

For 3D Euler equations in gas dynamics, planar vortex sheets are violently unstable, and this instability is analogous to the Kelvin-Helmholtz instability observed in incompressible fluids \cite{Serre}. In the context of 3D compressible MHD flows, Trakhinin \cite{Trakhinin2, Trakhinin} and Chen-Wang \cite{Wang} employed a distinct symmetrization approach to demonstrate the linear and nonlinear stability of compressible vortex sheets. Additionally, the stabilizing effects of magnetic fields on current vortex sheets have been shown in the incompressible case as well. Under the Syrovatskij stability condition \cite{Syrovatskij}, Morando-Trakhinin-Trebeschi \cite{Morando2} established a priori estimate with a loss of three-order derivatives for the linearized system. Trakhinin \cite{Trakhinin} proved a priori estimate without loss of derivatives from data for the linearized system with variable coefficients under a strong stability condition. Recently, Coulombel-Morando-Secchi-Trebeschi  \cite{Coulombel2} demonstrated a priori estimate without loss of derivatives for the nonlinear current-vortex sheet problem under the strong stability condition. The nonlinear stability of the incompressible current-vortex sheet problem was established by Sun-Wang-Zhang \cite{Sun} under the Syrovatskij stability condition. For 3D compressible isentropic elastic flows, Chen-Huang-Wang-Yuan \cite{chen1} proved the stabilizing effect of elasticity on 3D compressible vortex sheets.  For the 3D spacetime relativistic Euler equations, Chen-Paolo-Wang \cite{Paolo} obtained the necessary and sufficient condition for the weakly linear stability of relativistic vortex sheets in Minkowski spacetime and nonlinear stability of relativistic vortex sheets under small initial perturbations.

Although the linear stability of 2D rectilinear current vortex sheets has been established and a necessary and sufficient condition is provided \cite{Yu1}, the analogous issue in 3D appears to yield only a sufficient condition \cite{Trakhinin2}. In this paper, we focus on a very specific class of 3D current vortex sheets, where we can still derive a necessary and sufficient condition. We provide a rigorous confirmation that the magnetic field does not always enhance the stability of the compressible system under consideration. However, it is intriguing to note that the effects of compressibility are significantly mitigated when a transverse magnetic field is presented. The transverse magnetic field influences ideal MHD by altering the pressure contributions, which in turn modify the characteristic magnetosonic wave speed. Our analytical framework for the well-posedness of current vortex sheets is strongly inspired by the work of Morando-Secchi-Trebeschi \cite{Morando}.

 More precisely,  we consider the following compressible inviscid MHD equations in the whole plane $\mathbb{R}^{3}$ for time $t\geq0$:
\begin{equation}\label{1.1}
\left\{
\begin{aligned}
&\partial_{t} \rho + {\rm div} (\rho u)=0,\\
&\partial_{t}(\rho u) + {\rm div}(\rho u \otimes u- H \otimes H) + \nabla (p+\frac{|H|^{2}}{2})=0,\\
&\partial_{t} H- \nabla \times (u\times H)=0,
\end{aligned}
\right.
\end{equation}
where the functions $\rho$, $u=(u_{1},u_{2},u_{3})$, $p$ and $H=(H_{1},H_{2},H_{3})$
represent the fluid density, the velocity, the pressure and the magnetic field respectively. We assume that $p$ is a $C^{\infty}$ function of $\rho$, defined on $]0,\infty[$, and such that $p^{\prime}(\rho)>0$ for all $\rho$.  For the sake of clarity, we choose the pressure which satisfies the polytropic laws, i.e. $p(\rho)=A \rho^{\gamma}$ for $A>0$, $\gamma\geq 1$.  For convenience, we take $A=1$. The speed of sound $c(\rho)$  in the fluid is defined by the relation:
\begin{equation}\label{1.2}
\forall \rho >0, ~c(\rho)= \sqrt{p^{\prime}(\rho)}.
\end{equation}
The system \eqref{1.1} is supplemented by the divergence constraint
\begin{equation}\label{1.3}
{\rm div} H=0,
\end{equation}
this property holds at any time throughout the flow if it is satisfied initially.

Let $U(t,x_{1},x_{2},x_{3})= (\rho,u, H)(t,x_{1},x_{2},x_{3})$ be a solution of system \eqref{1.1} which is smooth on each side of  a surface  $\Gamma(t):= \{x_{3}=f(t,x_{1},x_{2}) \}$. Here, function $f$ describing the discontinuity front is part of the unknown of the problem, i.e., this is a free boundary problem and $x_{1},x_{2}$ are tangential coordinates. The whole space $\mathbb{R}^{3}$ can be separated by $\Gamma(t)$ into the upper domain $\Omega^+(t)$ and the lower domain $\Omega^-(t)$, which are defined by $$\Omega^+(t):=\{x_3>f(t,x_1,x_{2})\},$$
and  $$\Omega^-(t):=\{x_3<f(t,x_1,x_{2})\}.$$\\
We denote that
\begin{equation} \label{1.4}
	U=\left\{
	\begin{aligned}
		& U^{+}(t,x_{1},x_{2},x_{3}),  &x_3>f(t,x_{1},x_{2}) ,\\
		&U^{-}(t,x_{1},x_{2},x_{3}), &x_3<f(t,x_{1},x_{2}),
	\end{aligned}
	\right.
\end{equation}
where $U^{\pm}= (\rho^{\pm},u^{\pm}, H^{\pm})$.  To be weak solutions of \eqref{1.1} such piecewise smooth solutions should satisfy the  Rankine-Hugoniot conditions on $\Gamma(t)$:
\begin{equation}\label{1.5}
\left\{
\begin{aligned}
&[j]=0,\\
&j[u_{n}]+ [q]=0,~~~ j [u_{\tau}]= H_{n}[H_{\tau}],\\
&j [H_{\tau}/\rho] = H_{n} [u_{\tau}],\\
&[H_{n}]=0,
\end{aligned}
\right.
\end{equation}
where $j=\rho (\partial_{t} f -  u \cdot n)$ is the mass transfer flux across the discontinuity surface, $q=p+\frac{|H|^{2}}{2}$ is the total pressure, $n=(-\partial_{x_{1}} f, -\partial_{x_{2}} f, 1)$ is a normal vector to $\Gamma(t)$,  $\tau_{1}= (1 ,0,\partial_{1} f)$  and  $\tau_{2}= (0, 1,\partial_{2} f)$ are  tangential vectors to $\Gamma(t)$ . We also define these notation $H_{n}= <H,n>, u_{n}=<u,n>, H_{\tau_{i}}=<H,\tau_{i}>, i=1, 2,\ H_{\tau}=(H_{\tau_{1}},H_{\tau_{2}}), u_{\tau_{i}}=<u,\tau_{i}>, i=1,2, u_{\tau}=(u_{\tau_{1}},u_{\tau_{2}}) $ where $<,>$ means the inner product. Also, $[\phi]= \phi^{+}-\phi^{-}$ denotes the jump where the function $\phi$ across the hypersurface $\Gamma$.

Since we are interested in the Kelvin-Helmholtz instability  for  MHD flow, the mass does not transfer across the discontinuity interface $\Gamma(t)$:
\begin{equation}\label{1.6}
j^{\pm}=0, ~~{\rm on} ~~ \Gamma(t),
\end{equation}
now in view of the \eqref{1.6}, we can obtain
\begin{equation}\label{1.7}
[q]=0, \quad H^{\pm}_{n}=0, ~~{\rm on}~~\Gamma(t).
\end{equation}

Therefore, for 3D MHD Kelvin-Helmholtz instability,  the Rankine-Hugoniot conditions give the boundary conditions
\begin{equation}\label{1.8}
\partial_{t} f=u^{+}\cdot n=u^{-}\cdot n,~~p^{+}+ \frac{|H^{+}|^{2}}{2}=p^{-}+ \frac{|H^{-}|^{2}}{2},~~H^{+} \cdot n=H^{-} \cdot n=0,~~{\rm on}~~\Gamma(t).
\end{equation}

The system \eqref{1.1} is supplemented with the initial data 
\begin{equation}\label{1.9}
	\rho^{\pm}(0,x)=\rho^{\pm}_{0}(x),~ u^{\pm}(0,x)=u^{\pm}_{0}(x),~H^{\pm}(0,x)=H^{\pm}_{0}(x),~\text{in}~\Omega^{\pm}(0).
\end{equation}

%%%加\sigma 的定义%%%%%%%%%%%%%%%%%%%%%%%%%%%%%%%%%%%%%%%%%

Let $\sigma(p) = \log\rho(p)$, then, in terms of $(\sigma,u,H)$, \eqref{1.1} can be rewritten as 
\begin{align}\label{A1.1}
    \begin{cases}
        &\partial_{t} \sigma + (u\cdot \nabla)\sigma +\nabla\cdot u=0,\\
&\partial_{t}u + (u\cdot \nabla)u-\frac{1}{\rho}(H\cdot \nabla)H+c^2\nabla\sigma+\frac{1}{\rho}\nabla\frac{|H|^2}{2}=0,\\
&\partial_{t} H+(u\cdot \nabla)H+H\nabla\cdot u-(H\cdot \nabla )u=0.
    \end{cases}
\end{align}
where the speed of sound is considered as a function of $\sigma$, in fact, $c^2(\sigma):=\gamma {\rm{e}}^{\sigma(\gamma-1)}$.

The jump conditions \eqref{1.8} may be rewritten as 
\begin{equation} \label{1.11}
\begin{aligned}
e^{\sigma^{+}\gamma }+\frac{|H^{+}|^{2}}{2}=e^{\sigma^{-}\gamma}+\frac{|H^{-}|^{2}}{2},
u^{+}\cdot n= u^{-}\cdot n, ~H^{+} \cdot n=H^{-} \cdot n=0~~\mathrm{on}~~\Gamma(t).
\end{aligned}
\end{equation}

\subsection{Rectilinear solution}
\quad  It is easy to see that the system \eqref{1.1}-\eqref{1.7} admits  rectilinear solutions $\dot{U}=(\dot{f}, \dot{\rho}^{\pm},\dot{u}^{\pm},\dot{H}^{\pm})$ with the interface given by $\{x_{3}=0\}$ for all $t\geq0$. Then  $\Omega^{+}=\Omega^{+}(t)=\mathbb{R}^2\times (0,\infty)$ and  $\Omega^{-}=\Omega^{-}(t)=\mathbb{R}^2\times (-\infty,0)$ for all $t\geq0$.  More precisely, the front is flat, i.e., $\dot{f}=0$. To make sure the constant   density $\dot{\rho}^{\pm}$ satisfy the jump condition \eqref{1.7}, we must impose that 
\begin{equation} \label{1.12}
	\dot{\rho}^{+}=	\dot{\rho}^{-}:=\dot{\rho},
\end{equation}
where $\dot{\rho}$ is a positive constant. By using the Galilean transformation,  we can see that the upper fluid moves in the horizontal direction with some constant velocities and the lower fluid moves by the same constant velocities in the opposite direction, i.e., the  constant velocity field $\dot{u}^{\pm}$ has the following form:
\begin{equation}\label{1.13}
	\dot{u}=\left\{
	\begin{aligned}
		&(\dot{u}^{+}_{1},0,0),&x_3\geq0,\\
		&(\dot{u}^{-}_{1},0,0),&x_3<0,
	\end{aligned}
	\right.
\end{equation}\label{1.14}
where the constants $\dot{u}^{+}_{1},\dot{u}^{-}_{1}$ satisfy
\begin{equation}
	\dot{u}^{+}_{1}=-\dot{u}^{-}_{1}.
\end{equation}

Under the change of the scale of measurement, the constant transverse magnetic field $\dot{H}$ has the following form:
\begin{equation} \label{1.15}
	\dot{H}=\left\{
	\begin{aligned}
		\begin{aligned}
		&(0,\dot{H}^{+}_{2},0),&x_3\geq0,\\
		&(0,\dot{H}^{-}_{2},0),&x_3<0,
	\end{aligned}
	\end{aligned}
	\right.
\end{equation}
where the constants $\dot{H}^{+}_{2},\dot{H}^{-}_{2}$  satisfy
\begin{equation}\label{1.16}
|\dot{H}^{+}_{2}|=|\dot{H}^{-}_{2}|.
\end{equation}

\subsection{The new formulations}
\quad  Our analysis in this paper relies on the reformulation of the problem under consideration in new coordinates.  To begin with, we define the fixed domains $\Omega^{\pm}$ as
\begin{equation}\label{1.17}
	\begin{aligned}
		&\Omega^{+}:=\left\{x\in \mathbb{R}^{3}: x_{3}> 0\right\}, \\
		&\Omega^{-}:=\left\{x\in \mathbb{R}^{3}: x_{3}< 0\right\}.
	\end{aligned}
\end{equation}
Define the fixed  boundary $\Gamma$ as
$$
\Gamma:=\left\{x\in \mathbb{R}^{3}: x_{3}= 0\right\}.
$$
To reduce our free boundary problem to the fixed domains $\Omega^{\pm}$, we consider a change of variables on the whole space which maps back to the origin domains $(t,x)\mapsto (t,x_{1}, x_{2},x_{3}+ \psi(t,x))$. We construct such  $\psi$ by multiplying  the front $f$ with a smooth cut-off function depending on $x_{3}$:
\begin{equation}\label{1.18}
\psi(t,x)= \chi\left(\frac{x_{3}}{3(1+a)}\right)f(t,x_{1},x_{2}),~a=|f_{0}|_{L^{\infty}(\mathbb{R})},
\end{equation}
where $\chi\in C^{\infty}_{c}(\mathbb{R})$ is a smooth cut-off function with $0\leq \chi\leq 1$, $\chi(x_{3})=1$, for $|x_{3}|\leq 1$, $\chi(x_{3})=0$ for $|x_{3}|\geq 3$, and $ |\partial_{x_{3}} \chi(x_{3})|\leq 1$ for all $x_{3}\in \mathbb{R}$. We also assume $|f_{0}|_{L^{\infty}(\mathbb{R})}\leq 1$.  Moreover, we have
\begin{equation}\label{1.19}
	\begin{aligned}
		&\psi(x_{1},x_{2}, 0,t)=f(x_{1},x_{2},t), \\
		&\partial_{x_{3}} \psi(x_{1},x_{2}, 0,t )=0,\\
		&|\partial_{x_{3}} \psi | \leq \frac{1}{3(1+a)}|f|.
	\end{aligned}
\end{equation}

The change of variables that reduces the free boundary problem \eqref{1.1} to the fixed domains $\Omega^{\pm}$ is given in the following lemma.
\begin{lemm}
	Define the function $\Psi$ by
\begin{equation}\label{1.20}
	\Psi(t, x):=\left(x_{1}, x_{2},x_{3}+\psi(t, x)\right), \quad(t, x) \in[0, T] \times \Omega.
\end{equation}
Then  $\Psi: (t,x)\mapsto (t,x_{1}, x_{2},x_{3}+ \psi(t,x)) $ are diffeomorphism of $\Omega^{\pm}$ for all $t \in[0, T]$. 
\end{lemm}
\begin{proof} 
Since $|f_{0}|_{L^{\infty}(\mathbb{R})}\leq 1$, one can prove that there exists some $T>0$ such that $ \sup_{[0,T]} |f|_{L^{\infty}}<1$, the free interface is still a graph within the time interval $[0,T]$	and
	$$
	\begin{aligned}
		\partial_{3} \Psi_{3}(t, x) & =1+\partial_{3} \psi(t, x) \geq 1-\frac{1}{3}\times  2 = \frac{1}{3},
	\end{aligned}
	$$
which ensure that $\Psi: (t,x)\mapsto (t,x_{1},  x_{2}, x_{3}+ \psi(t,x)) $ are diffeomorphism of $\Omega^{\pm}$ for all $t \in[0, T]$. 
\end{proof}

We introduce the following operator notations
$$
\begin{array}{ll}
	A=[D \Psi]^{-1}=\left(\begin{array}{ccc}
		1 & 0  & 0\\
		0 & 1  & 0\\
		-\partial_{1} \psi / J & -\partial_{2} \psi / J  & 1 / J
	\end{array}\right), 
\end{array}
$$
and $ J=\operatorname{det}[D \Psi]=1+\partial_{3} \psi$.
Now we may reduce the free boundary problem \eqref{1.1} to a problem in the fixed domain $\Omega^{\pm}$ by the change of variables in Lemma 1.1. Set 
\begin{equation}\label{1.21}
	\begin{aligned}
		&v^{ \pm}(t, x):=u^{ \pm}(t, \Psi(t, x)),\quad  B^{ \pm}(t, x):=H^{ \pm}(t, \Psi(t, x)), \\
		&q^{ \pm}(t, x):=p^{ \pm}(t, \Psi(t, x)),\quad 
        \varrho^{ \pm}(t, x):=\rho^{\pm}(t, \Psi(t, x)),\\
        &h^{ \pm}(t, x):=\sigma^{\pm}(t, \Psi(t, x)).
	\end{aligned}
\end{equation}
Throughout the rest paper, an equation on $\Omega$ means that the equation holds in both $\Omega^{+}$ and $\Omega^{-}$.  For convenience, we consolidate notations by writing $v$, $B$, $q$, $\varrho$, 
$h$ to refer to $v^{\pm}$, $B^{\pm}$, $q^{\pm}$, $\varrho^{\pm}$, 
$h^{\pm}$ unless necessary.

Then the system \eqref{A1.1} and boundary conditions \eqref{1.11}  can be reformulated as: 
\begin{equation}\label{1.23}
	\begin{cases}
			\partial_{t} h+(\breve{v} \cdot \nabla) h+  \nabla^{\psi} \cdot v=0, & \text { in } \Omega,\\ 
		\partial_{t} v+(\breve{v}  \cdot \nabla) v+ c^{2}(h)\nabla^{\psi} h+ \frac{1}{ e^{h}} \nabla^{\psi} \frac{|B|^{2}}{2}=\frac{1}{ e^{h}} (\breve{B}\cdot \nabla ) B, & \text {  in } \Omega,\\
		\partial_{t} B+(\breve{v}  \cdot \nabla) B+B\nabla^{\psi} \cdot v = (\breve{B}\cdot \nabla ) v, & \text { in } \Omega, \\ 
		\nabla^{\psi} \cdot  B=0, & \text {  in } \Omega, \\ 
		\partial_{t} f=v\cdot n, & \text { on } \Gamma,\\
 \left[ e^{h\gamma}+ \frac{|B|^{2}}{2} \right]=0,~[v \cdot n]=0,~ B\cdot n=0, & \text { on } \Gamma,
	\end{cases}
\end{equation}
where $\nabla^{\psi}= A^{T} \nabla, \Delta^{\psi}= A^{T} \nabla \cdot  A^{T} \nabla,\ \breve{v} :=A v-\left(0, \partial_{t} \psi / J\right)$, $\breve{B} :=A B$  and $c^2(h) = \gamma {\rm{e}}^{h(\gamma-1)}$. 
	
Notice that
\begin{equation}\label{1.24}
	J=1, \quad \breve{v}_{3}=(v \cdot n-\partial_{t} \psi) / J=0,\quad \breve{B}_{3}=0, \quad \text { on } \Gamma.
\end{equation}

We are interested in the Kelvin-Helmholtz instability in MHD flows while the instability behavior mainly  happens on the boundary.  To see this,  we are going to derive the evolution equation of the front $f$ on the fixed boundary $\Gamma$. First, using the  momentum equation of  \eqref{1.23}, we deduce that 
\begin{equation}\label{1.25}
	\begin{aligned}
		\partial^{2}_{tt} f =&\partial_{t} v^{+}\cdot n + v^{+} \cdot \partial_{t} n \\
		=& -\left(\left(\breve{v}^{+} \cdot \nabla\right) v^{+}- \frac{1}{\varrho^{+}} \left(\breve{B}^{+}\cdot \nabla \right) B^{ +}+ c^{2}  \nabla^{\psi} h^{ +}+ \frac{1}{\varrho^{+}} \nabla^{\psi} \frac{|B^{+}|^{2}}{2} \right)\cdot n- v^{+} \cdot (\partial_{1}\partial_{t} f, \partial_{2}\partial_{t} f, 0)\\
		=& -v^{+}_{1}\partial_{1} v^{+}\cdot n-v^{+}_{2}\partial_{2} v^{+}\cdot n+\frac{1}{\varrho^{+}} B^{+}_{1}\partial_{1} B^{+}\cdot n +\frac{1}{\varrho^{+}} B^{+}_{2}\partial_{2} B^{+}\cdot n\\
		& -  c^{2} \nabla^{\psi} h^{+}\cdot n-   \frac{1}{\varrho^{+}}\nabla^{\psi} \frac{|B^{+}|^{2}}{2} \cdot n
		-v^{+}_{1} \partial_{1}\partial_{t} f-v^{+}_{2} \partial_{2}\partial_{t} f \\
		=& v^{+}_{1}\partial_{1}  n \cdot v^{+}- v^{+}_{1}\partial_{1} \partial_{t} f+v^{+}_{2}\partial_{2}  n \cdot v^{+}- v^{+}_{2}\partial_{2} \partial_{t} f-\frac{1}{\varrho^{+}}B^{+}_{1}\partial_{1} n \cdot B^{+}\\
		&-\frac{1}{\varrho^{+}}B^{+}_{2}\partial_{2} n \cdot B^{+}-  c^{2} \nabla^{\psi} h^{+}\cdot n-   \frac{1}{\varrho^{+}}\nabla^{\psi} \frac{|B^{+}|^{2}}{2} \cdot n
		-v^{+}_{1} \partial_{1}\partial_{t} f-v^{+}_{2} \partial_{2}\partial_{t} f \\
		=& -2v^{+}_{1}\partial_{1} \partial_{t} f -2v^{+}_{2}\partial_{2} \partial_{t} f  -  c^{2} \nabla^{\psi} h^{+}\cdot n -   \frac{1}{\varrho^{+}}\nabla^{\psi}\frac{|B^{+}|^{2}}{2} \cdot n- (v_{1}^{+})^{2} \partial^{2}_{11} f\\
		& -2v_{1}^{+}v_{2}^{+}  \partial^{2}_{12} f-(v_{2}^{+})^{2} \partial^{2}_{22} f+ \frac{1}{\varrho^{+}}(B_{1}^{+})^{2} \partial^{2}_{11} f
        + 2\frac{1}{\varrho^{+}}B_{1}^{+} B_{2}^{+} \partial^{2}_{12} f+ \frac{1}{\varrho^{+}}(B_{2}^{+})^{2} \partial^{2}_{22} f,\ \ \rm{on}\ \ \Gamma.
	\end{aligned}
\end{equation}
Similarly, we can also derive a evolution equation  from the negative part,
\begin{equation}\label{1.26}
	\begin{aligned}
		\partial^{2}_{tt} f
	=& -2v^{-}_{1}\partial_{1} \partial_{t} f -2v^{-}_{2}\partial_{2} \partial_{t} f  -  c^{2} \nabla^{\psi} h^{-}\cdot n -   \frac{1}{\varrho^{-}}\nabla^{\psi}\frac{|B^{-}|^{2}}{2} \cdot n- (v_{1}^{-})^{2} \partial^{2}_{11} f\\
& -2v_{1}^{-}v_{2}^{-}  \partial^{2}_{12} f-(v_{2}^{-})^{2} \partial^{2}_{22} f+ \frac{1}{\varrho^{-}}(B_{1}^{-})^{2} \partial^{2}_{11} f+ 2\frac{1}{\varrho^{-}}B_{1}^{-} B_{2}^{-} \partial^{2}_{12} f+ \frac{1}{\varrho^{-}}(B_{2}^{-})^{2} \partial^{2}_{22} f.
	\end{aligned}
\end{equation}

Therefore, summing up the $``+"$ equation \eqref{1.25} and $``-"$ equation \eqref{1.26}, we get
\begin{equation}\label{1.27}
	\begin{aligned}
		&\partial^{2}_{tt} f+ (v^{+}_{1}+v^{-}_{1})\partial_{1} \partial_{t} f+ (v^{+}_{2}+v^{-}_{2})\partial_{2} \partial_{t} f +\frac{1}{2}\left((c^{+})^{2} \nabla^{\psi} h^{+}\cdot n+(c^{-})^{2} \nabla^{\psi} h^{-}\cdot n\right)\\
		& + \frac{1}{2} \left(\frac{1}{\varrho^{+}}\nabla^{\psi} \frac{|B^{+}|^{2}}{2} \cdot n+  \frac{1}{\varrho^{-}} \nabla^{\psi} \frac{|B^{-}|^{2}}{2} \cdot n\right)+ \frac{1}{2}\left((v_{1}^{+})^{2}+(v_{1}^{-})^{2}\right) \partial^{2}_{11} f \\
		&+  (v_{1}^{+}v_{2}^{+}+v_{1}^{-}v_{2}^{-})  v_{1}^{-}v_{2}^{-}  \partial^{2}_{12} f
		+\frac{1}{2}((v_{2}^{+})^{2}+(v_{2}^{-})^{2}) \partial^{2}_{22} f
		\\
		&- \frac{1}{2}\left(\frac{1}{\varrho^{+}}(B_{1}^{+})^{2}+\frac{1}{\varrho^{-}}(B_{1}^{-})^{2}\right) \partial^{2}_{11} f-  \left(\frac{1}{\varrho^{+}}B_{1}^{+} B_{2}^{+} +\frac{1}{\varrho^{-}}B_{1}^{-} B_{2}^{-} \right)\partial^{2}_{12} f \\
		&- \frac{1}{2}\left(\frac{1}{\varrho^{+}}(B_{2}^{+})^{2}+\frac{1}{\varrho^{-}}(B_{2}^{-})^{2}\right) \partial^{2}_{22} f   =0, ~~~{\rm on}~ \Gamma.
	\end{aligned}
\end{equation}

\subsection{The wave equation for the log-density $h$ and the magnetic $B$}
\quad

Applying the operator $\partial_{t} +\breve{v} \cdot \nabla$ to the first equation of \eqref{1.23} and $\nabla^{\psi} \cdot$ to the second one, we obtain
\begin{equation}\label{1.28}
	\left\{
	\begin{aligned}
		&(\partial_{t} +\breve{v}\cdot \nabla)^{2} h+ (\partial_{t} +\breve{v}\cdot \nabla)\nabla^{\psi} \cdot v=0,\\
		&\nabla^{\psi} \cdot\left((\partial_{t} +\breve{v}\cdot \nabla)v \right)+ \nabla^{\psi} \cdot (c^{2}(h) \nabla^{\psi} h)+\nabla^{\psi} \cdot\left(\frac{1}{ e^{h}}\nabla^{\psi} \frac{|B|^{2}}{2}\right) = \nabla^{\psi} \cdot \left(\frac{1}{e^{h}} (\breve{B}\cdot \nabla)  B\right).
	\end{aligned}
	\right.
\end{equation}

Next, we take the difference of the two equations in \eqref{1.28} to deduce a wave-type equation:
\begin{equation}\label{1.29}
	(\partial_{t} +\breve{v}\cdot \nabla)^{2} h- c^{2}(h)\Delta^{\psi} h -\frac{1}{ e^{h}}\Delta^{\psi} \frac{|B|^{2}}{2} =\mathcal{F},
\end{equation}
where the term $\mathcal{F}=-[\partial_{t}+\breve{v}\cdot \nabla, \nabla^{\psi} ]v+ \nabla^{\psi} c^{2}(h) \cdot  \nabla^{\psi} h+\nabla^{\psi} \frac{1}{e^{h}}\cdot\nabla^{\psi} \frac{|B|^{2}}{2}-\nabla^{\psi} (\frac{1}{e^{h}} \breve{B})\cdot \nabla  B  =-\partial_{t} A_{ki} \partial_{k} v_{i}+  A_{ki} \partial_{k} \breve{v}_{j}\partial_{j} v_{i}- \breve{v}_{j}\partial_{j}A_{ki} \partial_{k}  v_{i}+ \nabla^{\psi} c^{2} \cdot  \nabla^{\psi} h+\nabla^{\psi} \frac{1}{e^{h}}\cdot\nabla^{\psi} \frac{|B|^{2}}{2}-\nabla^{\psi} (\frac{1}{ e^{h}} \breve{B})\cdot \nabla  B  $ is a lower order term in the second order differential equation for $h$.

Compared with the Euler flow, MHD flow contains $h$ and  $B$. Thus, we need to find the relationship between  $h$ and  $B$. To achieve this goal, by applying the operator  $\partial_{t} +\breve{v}\cdot \nabla$ to the third equation in  \eqref{1.23} and $\breve{B} \cdot \nabla $ to the second one  in  \eqref{1.23}, we get
\begin{equation}\label{1.30}
	\begin{aligned}
		&(\partial_{t} +\breve{v} \cdot \nabla)^{2} B - B(\partial_{t} +\tilde{v} \cdot \nabla)^{2} h+ c^{2} (\breve{B} \cdot \nabla) \nabla^{\psi} h\\
		&+  \frac{1}{e^{h}} (\breve{B} \cdot \nabla) \nabla^{\psi}\frac{|B|^{2}}{2} - \frac{1}{ e^{h}} (\breve{B} \cdot \nabla)^{2}B =\mathcal{G},
	\end{aligned}
\end{equation}
where the term $\mathcal{G}=(\partial_{t}+\breve{v} \cdot  \nabla) B \nabla^{\psi} \cdot v +[\partial_{t}+\breve{v} \cdot  \nabla, \breve{B}\cdot \nabla]u+\breve{B} \cdot \nabla c^{2}  \nabla^{\psi} h-\breve{B} \cdot \nabla\frac{1}{e^{h}}\tilde{B} \cdot \nabla B  +\breve{B}\cdot \nabla\frac{1}{e^{h}} \nabla^{\psi} \frac{|B|^{2}}{2} $ is the lower order term.

From the boundary conditions in \eqref{1.8}, we obtain
\begin{equation}\label{1.31}
	\left[ e^{\gamma h}+ \frac{|B|^{2}}{2}\right]=0,~~{\rm on}~~\Gamma,
\end{equation}
where $[h]= h^{+}_{|\Gamma}-h^{-}_{|\Gamma}$ represents the jump across $\Gamma$.

To determine the value of $h$, we add another condition involving the normal derivatives of $h^{\pm}$ on the boundary $\Gamma$. More precisely, taking the difference of  two equations \eqref{1.25} and  \eqref{1.26}, we can  obtain the jump of the normal derivatives $\nabla h^{\pm}\cdot n$,
\begin{equation}\label{1.32}
\begin{aligned}
   & \left[c^{2} \nabla^{\psi} h\cdot n+  \frac{1}{{\rm e}^h }\nabla^{\psi} \frac{|B|^{2}}{2} \cdot n \right]\\
    =&[-2v_1v_2\partial^2_{12}f-2v_{1}\partial_{1}\partial_{t} f-2v_{2}\partial_{2}\partial_{t} f- (v_{1})^{2} \partial^{2}_{11} f- (v_{2})^{2} \partial^{2}_{22} f\\
    &+\frac{1}{ e^{h}}(B_{1})^{2} \partial^{2}_{11} f+\frac{1}{e^{h}}(B_{2})^{2} \partial^{2}_{22} f+\frac{2}{e^{h}}B_{1}B_2\partial^{2}_{12} f],~~{\rm on}~\Gamma.
\end{aligned}	
\end{equation}

Combining \eqref{1.29},  \eqref{1.30}, \eqref{1.31} with \eqref{1.32} to obtain the system of   $h$
\begin{equation}\label{1.33}
	\left\{
	\begin{aligned}
		&	(\partial_{t} +\breve{v}\cdot \nabla)^{2} h- c^{2}\Delta^{\psi} h -\frac{1}{e^{h}}\Delta^{\psi} \frac{|B|^{2}}{2} =\mathcal{F},&~~{\rm  in}~\Omega,\\
		&(\partial_{t} +\breve{v} \cdot \nabla)^{2} B - B(\partial_{t} +\breve{v}\cdot \nabla)^{2} h+ c^{2} (\breve{B} \cdot \nabla) \nabla^{\psi} h\\
		&+  \frac{1}{ e^{h}} (\breve{B} \cdot \nabla) \nabla^{\psi}\frac{|B|^{2}}{2} - \frac{1}{ e^{h}} (\breve{B} \cdot \nabla)^{2}B =\mathcal{G},&~~{\rm  in}~\Omega, \\
		&\left[ e^{\gamma h}+ \frac{|B|^{2}}{2}\right]=0,&~~{\rm on}~\Gamma,\\
		&	\left[c^{2} \nabla^{\psi} h\cdot n+  \frac{1}{{\rm e}^h }\nabla^{\psi} \frac{|B|^{2}}{2} \cdot n \right]\\
    &=[-2v_1v_2\partial^2_{12}f-2v_{1}\partial_{1}\partial_{t} f-2v_{2}\partial_{2}\partial_{t} f- (v_{1})^{2} \partial^{2}_{11} f\\
    &\ \ \ \ \ - (v_{2})^{2} \partial^{2}_{22} f+\frac{1}{ e^{h}}(B_{1})^{2} \partial^{2}_{11} f+\frac{1}{e^{h}}(B_{2})^{2} \partial^{2}_{22} f+2\frac{1}{e^{h}}B_{1}B_2\partial^{2}_{12} f],&~~{\rm on}~\Gamma.
	\end{aligned}
	\right.
\end{equation}

Given a suitable stability condition for the unperturbed flows (i.e., the rectilinear solutions), we demonstrate the well-posedness of the vortex sheet problem for compressible MHD flows with a transverse magnetic field.

The rest of the paper is organized as follows. In Section 2, we linearize the equations \eqref{1.33} and \eqref{1.27} around a configuration (rectilinear solutions) and introduce the main Theorems 2.1 and 2.2.  The proof of Theorem 2.1 is presented in Section 3. In Section 4,  we analyze the root and the symbol of the pseudo-differential equation \eqref{3.9*} for the front. In Section 5, we identify a stable zone: the supersonic stable zone. Finally, in Section 6, the proof of Theorem 2.2 is completed.

\section{Linearized problems and main theorem}

\quad \quad In this section, we consider  a  linearized system in new coordinates. We are going to  construct special solutions for this linearized system.

\subsection{Construction of special  solution of the  linearized system.}
\quad \quad  It is easily verified that the particular solution in Euler coordinates is also a  particular solution in new coordinates such that
\begin{equation} \label{2.1}
	\dot{v}^{\pm}=\dot{u}^{\pm}=\left\{
	\begin{aligned}
		&(\dot{v}^{+}_{1},0,0),&x_3\ge 0,\\
		&(\dot{v}^{-}_{1},0,0),&x_3<0,
	\end{aligned}
	\right.
\end{equation}
and 
\begin{equation}\label{2.2}
	\dot{\varrho}^{\pm}=\dot{\rho}^{\pm}:=\dot{\varrho},
\end{equation}
and
\begin{equation} \label{2.3}
\dot{B}^{\pm}	=\dot{H}^{\pm}=\left\{
	\begin{aligned}
		&(0,\dot{B}^{+}_{2},0),&x_3\geq0,\\
		&(0,\dot{B}^{-}_{2},0),&x_3<0.
	\end{aligned}
	\right.
\end{equation}

\begin{rema}
From now on and throughout this paper, we use the new notations $(\dot{f},\dot{\varrho}^{\pm},\dot{v}^{\pm},\dot{B}^{\pm})$ to denote the rectilinear solutions  $(\dot{f},\dot{\rho}^{\pm},\dot{u}^{\pm},\dot{H}^{\pm})$, which are in fact the same constant quantities. Here we use the new notations to match the notations in the new coordinates.
\end{rema}

Now we will consider the constant coefficient linearized equations which are derived by linearization of equations \eqref{1.27} and \eqref{1.33} around the configuration: the constant velocity $\dot{v}^{\pm}=(\dot{v}^{\pm}_{1},0,0)$ along  the $x_{1}$-direction, the constant transverse magnetic field  $\dot{B}^{\pm}=(0,\dot{B}^{\pm}_{2},0)$ along the $x_{2}$-direction and flat front $\Gamma=\{x_{3}=0 \}$, the outer normal vector $n=(0,0,1)$. Moreover, all the rectilinear solutions can be transformed under the Galilean transformation and change of the scale of measurement to the following form:
\begin{equation}\label{2.4}
	\dot{v}^{+}_{1}+ \dot{v}^{-}_{1}=0, ~|\dot{B}^{+}_{2}|=|\dot{B}^{-}_{2}|.
\end{equation}
Therefore, we have the following linearized equations:
\begin{equation}\label{2.5}
	\begin{cases}
			\partial_{t} h+\dot{v}_{1} \partial_{1} h+  \nabla \cdot v=0, & \text { in } \Omega,\\ 
		\partial_{t} v+\dot{v}_{1} \partial_{1}  v+ c^{2} \nabla h+ \frac{\dot{B}_{2}}{\dot{\varrho}}   \nabla B_{2}=\frac{\dot{B}_{2}}{\dot{\varrho}} \partial_{2} B, & \text {  in } \Omega,\\
		\partial_{t} B+\dot{v}_{1} \partial_{1}  B+ \dot{B} \nabla \cdot v =\dot{B}_{2} \partial_{2}v, & \text { in } \Omega, \\ 
		\nabla \cdot B=0, & \text {  in } \Omega, \\ 
		\partial_{t} f=v_{3}- \dot{v}_{1} \partial_{1} f,  & \text { on } \Gamma,
	\end{cases}
\end{equation}
where $c^2:=c^{2}(\dot{\varrho})$.
In fact, let $v=\dot{v}+ \tilde{v}$ and $n=e_{3}+ \tilde{n}$, we linearize the origin boundary condition $ [v\cdot n]=0$ as follows:
\begin{equation*}
	[(\dot{v}+ \tilde{v})\cdot (e_{3}+ \tilde{n}) ]=[\tilde{v} \cdot e_{3}]+ [\dot{v}\cdot \tilde{n} ]+ [\tilde{v} \cdot \tilde{n}]=0,
\end{equation*}
where $\tilde{n}=(-\partial_{1}\tilde{f},-\partial_{2}\tilde{f},0 )$. Obviously, the third term is nonlinear term. Then the linear boundary condition is written as
\begin{equation*}
	[\tilde{v} \cdot e_{3}] =- [\dot{v}\cdot \tilde{n} ]= 2\dot{v}^{+}_{1}\partial_{1} \tilde{f}.
\end{equation*}
Similarly, we deduce that 
\begin{equation*}
	[ c^{2} h + \dot{\varrho}^{-1} \dot{B}_{2} {B}_{2} ]=0,~ \tilde{B}_{3} - \dot{B}_{2}\partial_{2} \tilde{f}=0.
\end{equation*}
Thus, the jump conditions on the boundary can be  linearized as follows:
\begin{equation}\label{2.6}
[ c^{2} h + \dot{\varrho}^{-1} \dot{B}_{2} {B}_{2} ]=0,~[v_{3}]= 2\dot{v}^{+}_{1}\partial_{1} f,~B_{3} - \dot{B}_{2}\partial_{2} f=0, 
\text{ on } \Gamma.
\end{equation}

We also get a linearized equation for the front $f$
\begin{equation}\label{2.7}
\begin{aligned}
\partial_{tt}^{2}f+ (\dot{v}^{+}_{1})^{2} \partial^{2}_{11} f+ \frac{1}{2}  c^{2} (\partial_{3} h^{+}+  \partial_{3} h^{-} )+ \frac{1}{2} \left(  \frac{\dot{B}^{+}_{2}}{\dot{\varrho}}\partial_{3} B^{+}_{2} +  \frac{\dot{B}^{-}_{2}}{\dot{\varrho}}\partial_{3} B^{-}_{2}\right) =\frac{(\dot{B}_{2}^{+})^{2}} {\dot{\varrho}} \partial^{2}_{22} f, ~~~\text { on }~ \Gamma,
\end{aligned}
\end{equation}
and a linearized system for  the pressure $h$
\begin{equation}\label{2.8}
\left\{
\begin{aligned}
&(\partial_{t} +\dot{v}_{1} \partial_{1})^{2} h-  c^{2}\Delta h-\frac{\dot{B}_{2}}  {\dot{\varrho}} \Delta B_{2}=\mathcal{K},&~~~\text {  in }~ \Omega,\\
& (\partial_{t} +\dot{v}_{1} \partial_{1})^{2} B_{2}  - \dot{B}_{2} (\partial_{t} +\dot{v}_{1} \partial_{1})^{2} h + c^{2}\dot{B}_{2} \partial_{22}^{2} h =0&~~~\text {  in }~ \Omega,\\
&[ c^{2} h + \dot{\varrho}^{-1} \dot{B}_{2} {B}_{2}  ]=0,&~~~\text { on }~ \Gamma,\\
&[c^{2} \partial_{3}h +\dot{\varrho}^{-1} \dot{B}_{2} \partial_{3} B_{2} ]= -4\dot{v}^{+}_{1} \partial_{t}\partial_{1}f,&~~~\text { on }~ \Gamma.
\end{aligned}
\right.
\end{equation}
where $\mathcal{K}$ is a given source term.

Similar to the result in \cite{Shivamoggi}, since we want to construct special solutions of the linear systems \eqref{2.5}-\eqref{2.8},  we assume the solution $U=(f,h,v,B)(t,x_{1},x_{3})$ is independent of the $x_{2}$-direction. Therefore, regardless of the $x_{2}$-direction derivative, the system  \eqref{2.5}-\eqref{2.8} reduces to the following one:
\begin{equation}\label{2.10}
	\begin{cases}
			\partial_{t} h+\dot{v}_{1} \partial_{1} h+  \nabla \cdot v=0, & \text {  in } \Omega,\\ 
		\partial_{t} v+\dot{v}_{1} \partial_{1}  v+ c^{2} \nabla h+ \frac{\dot{B}_{2}}{\dot{\varrho}} \nabla B_{2}=0, & \text { in } \Omega,\\
		\partial_{t} B+\dot{v}_{1} \partial_{1}  B+ \dot{B} \nabla \cdot v =0, & \text {  in } \Omega, \\ 
		\nabla \cdot B=0, & \text {  in } \Omega, \\ 
		\partial_{t} f=v_{3}- \dot{v}_{1} \partial_{1} f,  & \text { on } \Gamma,\\
        [ c^{2} h + \dot{\varrho}^{-1}\dot{B}_{2} {B}_{2} ]=0,~[v_{3}]= 2\dot{v}^{+}_{1}\partial_{1} f,~B_{3} =0,  & \text { on } \Gamma,
	\end{cases}
\end{equation}
and
\begin{equation}\label{2.11}
\begin{aligned}
\partial_{tt}^{2}f+ (\dot{v}^{+}_{1})^{2} \partial^{2}_{11} f+ \frac{1}{2}  c^{2} (\partial_{3} h^{+}+  \partial_{3} h^{-} )+ \frac{1}{2} \left(  \frac{\dot{B}^{+}_{2}}{\dot{\varrho}}\partial_{3} B^{+}_{2} +  \frac{\dot{B}^{-}_{2}}{\dot{\varrho}}\partial_{3} B^{-}_{2}\right) =0, ~~~\text { on }~ \Gamma,
\end{aligned}
\end{equation}
and a linearized system for  the pressure $h$
\begin{equation}\label{2.12}
\left\{
\begin{aligned}
&(\partial_{t} +\dot{v}_{1} \partial_{1})^{2} h-  c^{2}(\partial^{2}_{11}+\partial^{2}_{33} ) h-  \frac{\dot{B}_{2}}  {\dot{\varrho}} (\partial^{2}_{11}+\partial^{2}_{33} ) B_{2}=\mathcal{K},&~~~\text {  in }~ \Omega,\\
& (\partial_{t} +\dot{v}_{1} \partial_{1})^{2} B_{2}  - \dot{B}_{2} (\partial_{t} +\dot{v}_{1} \partial_{1})^{2} h =0,&~~~\text {  in }~ \Omega,\\
&[ c^{2} h + \dot{\varrho}^{-1}\dot{B}_{2} {B}_{2}  ]=0,&~~~\text { on }~ \Gamma,\\
&[c^{2} \partial_{3}h +\dot{\varrho}^{-1} \dot{B}_{2} \partial_{3} B_{2} ]= -4\dot{v}^{+}_{1} \partial_{t}\partial_{1}f,&~~~\text { on }~ \Gamma.
\end{aligned}
\right.
\end{equation}

%%%%%%%%%%%%%%%%%%%%%%%%%%%%%%%%%%%%%%%%%%%%%%

For $\gamma\geq 1$, we introduce $\tilde{f}:= e^{-\gamma t} f$, $\tilde{h}:=  e^{-\gamma t} h$, $\tilde{B_{2}}:=  e^{-\gamma t} B_{2}$, $\tilde{\mathcal{K}}:=e^{-\gamma t}\mathcal{K}$ and consider the equations
\begin{equation}\label{3.5*}
\begin{aligned}
(\gamma+\partial_{t})^{2}\tilde{f}+ (\dot{v}_1^+)^2\partial_{11}^2 \tilde{f}+\frac{1}{2}  c^{2} (\partial_{3}  \tilde{h}^{+}+  \partial_{3}  \tilde{h}^{-} )+ \frac{1}{2} \left(  \frac{\dot{B}^{+}_{2}}{\dot{\varrho}}\partial_{3} \tilde{B}^{+}_{2} +  \frac{\dot{B}^{-}_{2}}{\dot{\varrho}}\partial_{3} \tilde{B}^{-}_{2}\right) =0,~~\text{on}~\Gamma,
\end{aligned}
\end{equation}
and
\begin{equation}\label{3.6*}
\left\{
\begin{aligned}
&(\gamma+\partial_{t} +\dot{v}_{1} \partial_{1})^{2} \tilde{h}-  c^{2}(\partial^{2}_{11}+\partial^{2}_{33} ) \tilde{h}-  \frac{\dot{B}_{2}}  {\dot{\varrho}} (\partial^{2}_{11}+\partial^{2}_{33} )  \tilde{B}_{2}=\tilde{\mathcal{K}}, &~~~\text{ in}~\Omega,\\
& (\gamma+\partial_{t} +\dot{v}_{1} \partial_{1})^{2} \tilde{B}_{2}  - \dot{B}_{2} (\gamma+\partial_{t} +\dot{v}_{1} \partial_{1})^{2} \tilde{h}=0, &~~~\text{ in}~\Omega,\\
&[c^{2}\tilde{h}+ \dot{\varrho}^{-1}\dot{B}_{2} \tilde{B}_{2}]= 0, &~~~\text{on}~ \Gamma,\\
&[c^{2} \partial_2 \tilde{h}+ \dot{\varrho}^{-1} \dot{B}_{2} \partial_{3} \tilde{B}_{2}]=-4\dot{v}_1^{+}(\gamma+\partial_t)\partial_1 \tilde{f}, &~~~\text{on}~ \Gamma.
\end{aligned}
\right.
\end{equation}
 Let us denote the Fourier transform of $\tilde{f}, \tilde{h}, \tilde{B}_{2}, \tilde{\mathcal{K}}$ by $\hat{f}, \hat{h}, \hat{B}_{2}, \hat{\mathcal{K}}$ with respect to $(t,x_{1})$ and the dual variables are denoted by $(\delta,\eta)$. Set $\tau=\gamma+i \delta$, we have the following result.
%%%%%%%%%%%%%这个定理我不会叙述，需更改
\begin{theo}
Let $\tilde{\mathcal{K}}$ satisfy
\begin{equation}\label{3.7*}
	\lim_{x_{3} \rightarrow +\infty} \hat{\mathcal{K}}(\cdot, \pm x_{3}) =0.
\end{equation}
Assume that $\tilde{f}$, $\tilde{h}$, $\tilde{B}_{2} $ is a solution of \eqref{3.5*}, \eqref{3.6*} with
\begin{equation}\label{3.8*}
	\lim_{x_{3} \rightarrow +\infty} \hat{h}^{\pm}(\cdot, \pm x_{3}) =0.
\end{equation}
Then $\hat{f}$ solves the second order psedo-differential equation
\begin{equation}\label{3.9*}
 \left(\tau^{2}-(\dot{v}^{+}_{1})^{2} \eta^{2}-2i\tau\eta  \dot{v}^{+}_{1} \frac{\mu^{+}-\mu^{-}}{\mu^{+}+\mu^{-}}\right)\hat{f}+ \frac{\mu^{+}\mu^{-}}{\mu^{+}+ \mu^{-}} W=0,
 \end{equation}
where $\mu^{\pm}=\sqrt{\frac{(\tau +i \dot{v}^{\pm}_{1}\eta)^{2}}{(\bar{C}_{B})^{2} }+ \eta^{2}}$ such that $\mathfrak{R} \mu^{\pm}>0$ if $\mathfrak{R} \tau>0$, $\bar{C}_{B}^{2}= \bar{c}^{2} + \bar{c}_{A}^{2}$ and
\begin{equation}\label{3.10*}
	W= W(\tau,\eta)= \frac{1}{\mu^{+}(\bar{C}_{B})^{2}} \int_{0}^{\infty} e^{- \mu^{+} y} \hat{\mathcal{K}}^{+}(\cdot,y) d y-\frac{1}{\mu^{-}(\bar{C}_{B})^{2}} \int_{0}^{\infty} e^{- \mu^{-} y} \hat{\mathcal{K}}^{-}(\cdot,y) d y.
 \end{equation}	
\end{theo}
From the formula of the roots $\mu^{\pm}$,  we know that  $\mu^{\pm}$ are homogenous functions of degree 1 in $(\tau,\eta)$. It follows that the ratio $\frac{\mu^{+}-\mu^{-}}{\mu^{+}+\mu^{-}}$ is homogeneous of degree 0. Thus we deduce that  the symbol of \eqref{3.9*} is a homogeneous function of degree 2. For this  second order  psedo-differential equation for $f$ we can obtain the following theorem.
%%%%%%%%%%%%%%%%%%%%%%%%%%%%%%下面定理与记号都没改动...不清楚
\begin{theo}
Assume  $M_{B}:= \frac{\dot{v}_1^{+}}{\bar{C}_{B}}>\sqrt{2}$
and $\mathcal{K}^{-}\in L^{2}(\mathbb{R}^{-}; H^{s}_{\gamma}(\mathbb{R}^{2}))$, $\mathcal{K}^{+}\in L^{2}(\mathbb{R}^{+}; H^{s}_{\gamma}(\mathbb{R}^{2}))$. There exists a unique solution $f\in H^{s+1}_{\gamma}(\mathbb{R}^{2})$ of the equation \eqref{3.9*}, satisfying the estimate
\begin{equation}\label{3.13*}
\gamma^{3} \|f\|^{2}_{H^{s+1}_{\gamma}(\mathbb{R}^{2})}\leq  C(\|\mathcal{K}^{+}\|^2_{L^{2}(\mathbb{R}^{+}; H^{s}_{\gamma}(\mathbb{R}^{2}))}+
\|\mathcal{K}^{-}\|^2_{L^{2}(\mathbb{R}^{-}; H^{s}_{\gamma}(\mathbb{R}^{2}))} ), \quad ~~for~~all ~\gamma\geq 1,
\end{equation}
for a suitable constant $C>0$ independent of $\mathcal{K}^{\pm}$ and $\gamma$.
\end{theo}

\begin{rema}
		 In Euler fluids, the velocity of the rectilinear solutions (piecewise constant solutions) must exceed $\sqrt{2} c$ to prevent the occurrence of Kelvin-Helmholtz instability. However, in this article, we demonstrate that the velocity of rectilinear solutions needs to be greater than $\sqrt{2} \bar{C}_{B}$ to suppress Kelvin-Helmholtz instability, here $\bar{C}_{B} >c$. In other words, we provide a rigorous confirmation that the magnetic field does not always enhance the stability of the compressible system under consideration. 
	\end{rema}
	
  \begin{rema}
	In this paper, we consider the limit as $\dot{B}_{2}$ approaches zero. We observe that the stability condition $M_{B}>\sqrt{2}$ reduces to $M>\sqrt{2}$, which is a standard condition in Euler fluids. Consequently, the result referenced in \cite{Morando} can be obtained.
 \end{rema}

 \begin{rema}
	From the form of rectilinear solution, we can see that the magnetic field is transversal to the velocity, meanwhile parallel to the boundary ${x_{3}=0}$.
 \end{rema}

\begin{nota}
	For all $s\in \mathbb{R}^{2}$ and for all $\gamma\geq 1$, the usual Sobolev space $H^{s}(\mathbb{R}^{2})$ is equipped with the following norm:
	\begin{equation*}
		\| v\|^{2}_{s,\gamma}:= \frac{1}{(2\pi)^{2}} \int\int_{\mathbb{R}^{2}} \Lambda^{2s} (\tau,\eta) |\hat{v}(\delta,\eta)|^{2}, ~\Lambda^{2} := (\gamma^{2}+ \delta^{2}+\eta^{2})^{\frac{s}{2}}= (|\tau|^{2}+ \eta^{2})^{\frac{s}{2}},
	\end{equation*}
	where $\hat{v}(\delta,\eta)$ is the Fourier transform of $v(t,x_{1})$ and $\tau=\gamma+i \delta$. For $s\in \mathbb{R}$ and $\gamma\geq 1$, we introduce the weighted Sobolev space $H^{s}_{\gamma}(\mathbb{R}^{2})$ as
	\begin{equation*}
		H^{s}_{\gamma}(\mathbb{R}^{2}) : = \{u \in \mathcal{D}^{\prime}(\mathbb{R}^{2}) : e^{-\gamma t} u(t,x_{1}) \in H^{s}(\mathbb{R}^{2})\},
	\end{equation*}
	and its norm $\|u\|_{H^{s}_{\gamma}(\mathbb{R}^{2})}:= \|e^{-\gamma t} u \|_{s,\gamma}$.  We write $L^{2}_{\gamma}(\mathbb{R}^{2}):= H^{0}_{\gamma}(\mathbb{R}^{2})$ and $\|u\|_{L^{2}_{\gamma}(\mathbb{R}^{2})}:= \|e^{-\gamma t} u \|$.
	
	We define $L^{2}(\mathbb{R}^{\pm}: H^{s}_{\gamma}(\mathbb{R}^{2}))$ as the spaces of distributions with finite norm
	\begin{equation*}
		\| u \|^{2}_{L^{2}(\mathbb{R}^{\pm}: H^{s}_{\gamma}(\mathbb{R}^{2}))}:= \int_{\mathbb{R}^{+}} \|u(\cdot, \pm x_{3})\|^{2} _{H^{s}_{\gamma}(\mathbb{R}^{2})} dx_{3}.
	\end{equation*}
\end{nota}

\section{Proof of Theorem 2.1}
%%%%%%%%%%%%%%%%%%%%%%%%%%%%%%%%%%%%%%%%%%%%%%

\quad \quad Following the approach in \cite{Morando}, we take the Fourier-Laplace transform of problems \eqref{3.5*} and \eqref{3.6*}  to deduce the formula for $\partial_{3}\hat{h}^{+}(0) + \partial_{3}\hat{h}^{-}(0)$. We then substitute this formula into \eqref{3.5*}, resulting in a second-order wave-type equation for $f$. More precisely, we begin by taking the Fourier transform of the problems \eqref{3.5*} and \eqref{3.6*} to obtain
\begin{equation}\label{4.1*}
\begin{aligned}
\tau^{2}\hat{f}-(\dot{v}^{+}_{1})^{2} \eta^{2} \hat{f}+ +\frac{1}{2}  c^{2} (\partial_{3}  \hat{h}^{+}+  \partial_{3}  \hat{h}^{-} )+ \frac{1}{2} \left(  \frac{\dot{B}^{+}_{2}}{\dot{\varrho}}\partial_{3} \hat{B}^{+}_{2} +  \frac{\dot{B}^{-}_{2}}{\dot{\varrho}}\partial_{3} \hat{B}^{-}_{2}\right)=0,~~\text{on}~\Gamma,
\end{aligned}
\end{equation}
and
\begin{equation}\label{4.2*}
\left\{
\begin{aligned}
&(\tau +i \dot{v}_{1} \eta)^{2} \hat{h}+ c^{2} \eta^{2} \hat{h} -c^{2}   \partial^{2}_{33} \hat{h}+  \frac{\dot{B}_{2}}  {\dot{\varrho}} \eta^{2}\hat{B}_{2}   - \frac{\dot{B}_{2}}  {\dot{\varrho}}  \partial^{2}_{33}   \hat{B}_{2}=\hat{\mathcal{K}},&~~~\text{in}~\Omega,\\
& (\tau +i \dot{v}_{1}\eta)^{2} \hat{B}_{2}  - \dot{B}_{2} (\tau +i \dot{v}_{1}\eta)^{2} \hat{h}=0, &~~~\text{in}~\Omega,\\
&[c^{2}\hat{h}+ \dot{\varrho}^{-1}\dot{B}_{2} \hat{B}_{2}]= 0, &~~~\text{on}~ \Gamma,\\
&[c^{2} \partial_2 \hat{h}+ \dot{\varrho}^{-1} \dot{B}_{2} \partial_{3} \hat{B}_{2}]=-4i \tau \eta  \hat{f} \dot{v}_1^{+}, &~~~\text{on}~ \Gamma.
\end{aligned}
\right.
\end{equation}
From the first two equations in  \eqref{4.2*}, we can deduce  $$\hat{h}=\hat{B}_{2},$$
then we substitute this relationship into \eqref{4.2*} to get
\begin{equation}\label{4.1**}
\begin{aligned}
\tau^{2}\hat{f}-(\dot {v}^{+}_{1})^{2} \eta^{2} \hat{f}+ \frac{ \bar{C}_{B}^{2}}{2}[\partial_{3} \hat{h}^{+}+  \partial_{3} \hat{h}^{-}]=0,~~\text{on}~\Gamma,
\end{aligned}
\end{equation}
and 
\begin{equation}\label{4.3*}
\left\{
\begin{aligned}
&(\tau +i \dot{v}_{1}\eta)^{2}\hat{h} +  \bar{C}_{B}^{2} \eta^{2} \hat{h}- \bar{C}_{B}^{2} \partial^{2}_{2} \hat{h}=\hat{\mathcal{K}},&~~~\text{on}~\Omega,\\
&[\bar{C}_{B}^{2} \hat{h}]=0,&~~~\text{on}~ \Gamma,\\
&[\bar{C}_{B}^{2}  \partial_{3}\hat{h}  ]= -4i \tau \eta \hat{f}\dot{v}^{+}_{1}, &~~~\text{on}~ \Gamma,
\end{aligned}
\right.
\end{equation}
where
\begin{equation}\label{4.4*}
\bar{C}_{B}^{2}:= c^{2} + c_{A}^{2},~ c_{A}^{2}:=\frac{(\dot{B}_{2})^{2}}  {\dot{\varrho}}  .
\end{equation}

Next, we define the Laplace transform in $x_{3}$ with dual variable $s\in \mathbb{C}$ as follow:
\begin{equation*}
\mathcal{L}[\hat{h}](s)=\int_{0}^{\infty} e^{-s x_{3}} \hat{h}(\cdot, x_{3}) d x_{2},
\end{equation*}
\begin{equation*}
\mathcal{L}[\hat{\mathcal{K}}](s)=\int_{0}^{\infty} e^{-s x_{3}} \hat{\mathcal{K}}(\cdot, x_{3}) d x_{3}.
\end{equation*}

Performing the Laplace transform of the problem \eqref{4.3*} to get
\begin{equation}\label{4.5}
[(\tau +i \dot{v}_{1}\eta)^{2}+ \bar{C}_{B}^{2}\eta^{2}  -  \bar{C}_{B}^{2}s^{2}]\mathcal{L}[\hat{h}](s) = \mathcal{L}[\hat{\mathcal{F}}](s)-\bar{C}_{B}^{2} s \hat{h}(0) \mp \bar{C}_{B}^{2} \partial_{3} \hat{h}(0),
\end{equation}
from this equation \eqref{4.4*}, we can obtain
\begin{equation}\label{4.6}
\begin{aligned}
\mathcal{L}[\hat{h}](s)= \frac{\bar{C}_{B}^{2} s \hat{h}(0) \pm \bar{C}_{B}^{2} \partial_{3} \hat{h}(0)}{\bar{C}_{B}^{2} s^{2}-(\tau +i \dot{v}_{1}\eta)^{2}- \bar{C}_{B}^{2}\eta^{2}} - \frac{\mathcal{L}[\hat{\mathcal{K}}](s)}{ \bar{C}_{B}^{2} s^{2}-(\tau +i \dot{v}_{1}\eta)^{2}- \bar{C}_{B}^{2}\eta^{2}}.
\end{aligned}
\end{equation}

We denote the root of the equation (in s)
\begin{equation}\label{4.7}
 \bar{C}_{B}^{2} s^{2}-(\tau +i \dot{v}_{1}\eta)^{2}- \bar{C}_{B}^{2}\eta^{2}=0,
\end{equation}

\begin{equation}\label{4.8}
\mathfrak{R} \mu^{\pm} >0,~if~\gamma>0.
\end{equation}

By taking the inverse Laplace transform of \eqref{4.6},  for $x_{3}>0$, we  have
\begin{equation}\label{4.9}
\hat{h}^{+}(\cdot,x_{3})= \hat{h}^{+}(0) {\rm cosh}(\mu^{+}x_{3}) + \partial_{3}\hat{h}^{+}(0)\frac{{\rm sinh}(\mu^{+} x_{3})}{\mu^{+}}- \int^{x_{3}}_{0} \frac{{\rm sinh}(\mu^{+} (x_{3}-y))}{\bar{C}_{B}^{2}\mu^{+}} \hat{\mathcal{K}}^{+}(\cdot,y) dy,
\end{equation}

\begin{equation}\label{4.10}
\hat{h}^{-}(\cdot,-x_{3})= \hat{h}^{-}(0) {\rm cosh}(\mu^{-}x_{3}) - \partial_{3}\hat{h}^{-}(0)\frac{{\rm sinh}(\mu^{-} x_{3})}{\mu^{-}}- \int^{x_{3}}_{0} \frac{{\rm sinh}(\mu^{-} (x_{3}-y))}{\bar{C}_{B}^{2}\mu^{-}} \hat{\mathcal{K}}^{-}(\cdot, -y) dy.
\end{equation}

In order to determine the values of $\hat{h}^{+}(0)$, $\partial_{3}\hat{h}^{+}(0)$ in \eqref{4.9} and \eqref{4.10}, we need two addition conditions. To achieve this goal,   in according with the assumption \eqref{3.7*},  it follows that
\begin{equation}\label{4.11}
\lim_{x_{3}\rightarrow \infty} \int_{0}^{x_{3}}  e^{-\mu^{\pm} (x_3-y)} \hat{\mathcal{K}}^{\pm}(\cdot,\pm y) dy=0.
\end{equation}

Using the equations \eqref{4.6},  \eqref{4.7},  \eqref{4.8} and  \eqref{3.7*},  we deduce
\begin{equation}\label{4.12}
\hat{h}^{\pm}(0) \pm \frac{1}{\mu^{\pm}} \partial_{3}\hat{h}^{\pm}(0)- \frac{1}{(\bar{C}_{B})^2 \mu^{\pm}} \int^{\infty}_{0} e^{-\mu^{\pm} y} \hat{\mathcal{K}}^{+}(\cdot, \pm y) dy =0.
\end{equation}

Combing two boundary conditions in \eqref{4.3*} and two boundary conditions in \eqref{4.9},  we have following system:
\begin{equation}\label{4.13}
\left\{
\begin{aligned}
& \bar{C}_{B}^{2} \hat{h}^{+}(0)- \bar{C}_{B}^{2}  \hat{h}^{-}(0)=0,\\
&\bar{C}_{B}^{2} \partial_{3}\hat{h}^{+}(0)-\bar{C}_{B}^{2}\partial_{3}\hat{h}^{-}(0)= - 4i \tau\eta \hat{f}\dot{v}_{1}^{+} ,\\
&\mu^{+} \hat{h}^{+}(0) +\partial_{3}\hat{h}^{+}(0)=  \frac{1}{\bar{C}_{B}^{2}}\int^{\infty}_{0} e^{-\mu^{+} y} \hat{\mathcal{K}}^{+}(\cdot,y) dy,\\
&\mu^{-} \hat{h}^{-}(0) -\partial_{3}\hat{h}^{-}(0)= \frac{1}{\bar{C}_{B}^{2}} \int^{\infty}_{0} e^{-\mu^{-} y} \hat{\mathcal{K}}^{-}(\cdot, -y) dy.
\end{aligned}
\right.
\end{equation}

From \eqref{4.8}, we know that the value of $\mu^{+}+\mu^{-}$ never vanishes as long as $
\gamma>0$. By directly computing, we deduce
%%%%%%%%%下式可能计算有误，需讨论！
\begin{equation}\label{4.14}
\bar{C}_{B}^{2}\partial_{3}\hat{h}^{+}(0)+\bar{C}_{B}^{2}\partial_{3}\hat{h}^{-}(0)=-4i \tau\eta \hat{f}\dot{v}^{+}_{1} \frac{\mu^{+}-\mu^{-}}{\mu^{+}+\mu^{-}}+ 2 \frac{\mu^{+}\mu^{-}}{\mu^{+}+\mu^{-}} W,
\end{equation}
where
\begin{equation}\label{3.16}
W = \frac{1}{\mu^{+}\bar{C}_{B}^{2}} \int_{0}^{\infty} e^{-\mu^{+}y} \hat{\mathcal{K}}^{+}(\cdot,y) dy -  \frac{1}{\mu^{-}\bar{C}_{B}^{2}} \int_{0}^{\infty} e^{-\mu^{-}y} \hat{\mathcal{K}}^{-}(\cdot,-y) dy.
\end{equation}

Substituting \eqref{4.14} into \eqref{4.1*} to obtain a second-order wave equation for $\hat{f}$
\begin{equation}\label{3.155}
\begin{aligned}
\left(\tau^{2}-(\dot{v}^{+}_{1})^{2} \eta^{2}-2i \tau\eta \dot{v}^{+}_{1} \frac{\mu^{+}-\mu^{-}}{\mu^{+}+\mu^{-}}\right)\hat{f}+ \frac{\mu^{+} \mu^{-}}{\mu^{+}+\mu^{-}}W=0,~~\text{on}~\Gamma.
\end{aligned}
\end{equation}

\section{The symbol of equation \eqref{3.155} for the front}
\subsection{Study of the roots $\mu^{\pm}$}
\quad \quad To further simplify the equation, we need to study the roots $\mu^{\pm}$. We introduce  the symbol of \eqref{3.155} by $\Sigma$:
\begin{equation}\label{4.144}
	\Sigma= \tau^{2}-(\dot{v}^{+}_{1})^{2} \eta^{2}-2i\tau \eta\dot{v}^{+}_{1} \frac{\mu^{+}-\mu^{-}}{\mu^{+}+\mu^{-}},
\end{equation}
and define a hemisphere by  $\Xi_1$:
\begin{equation}
	\Xi_1= \{ (\tau,\eta)\in \mathbb{C}\times \mathbb{R} : |\tau|^{2}+ \eta^{2}=1, \mathfrak{R} \tau\geq 0 \},
\end{equation}
and the set of "frequencies"
\begin{equation}
    \Xi= \{ (\tau,\eta)\in \mathbb{C}\times \mathbb{R} : \mathfrak{R} \tau\geq 0, (\tau,\eta)\neq(0,0)\}=(0,\infty)\cdot \Xi_1.
\end{equation}
To begin with,  we are going to study the property of the roots $\mu^{\pm}$.
\begin{lemm}
	(\cite{Morando}) Let $(\tau,\eta)\in \Xi$ and let us consider the equation
\begin{equation}\label{4.15}
	s^{2}= \frac{(\tau+i \dot{v}^{\pm}_{1} \eta)^{2}}{\bar{C}^{2}_{B}} + \eta^{2}.
\end{equation}
For both cases $\pm$ of \eqref{4.15} there exists one root, denoted by $\mu^{\pm}=\mu^{\pm}(\tau,\eta)$, such that $\mathfrak{R}\mu^{\pm}>0$ as long as $\mathfrak{R}\tau>0$. The other root is $-\mu^{\pm}$. The root $\mu^{\pm}$ admit a continuous extension to points $(\tau,\eta)=(i \delta, \eta)\in \Xi$, i.e., with $\mathfrak{R} \tau=0$. More precisely, 
\begin{item}
	\item(i) When $\eta=0$, we have $\mu^{\pm}(i\delta,0)= i\delta/\bar{C}_{B}$.
	\item(ii) When $\eta\neq0$,  it is noticed that $\Sigma(\tau, \eta)= \Sigma(\tau, -\eta)$, thus we only consider positive values of $\eta$ in all this section: $\eta> 0$, if $- \frac{ \dot{v}^{\pm}_{1}}{\bar{C}_{B}}-1< \frac{\delta}{\bar{C}_{B} \eta} < - \frac{ \dot{v}^{\pm}_{1}}{\bar{C}_{B}}+1$,  we have
		\begin{equation}\label{4.16}
			\mu^{\pm} (i\delta,\eta)= \sqrt{\frac{- (\delta+ \dot{v}^{\pm}_{1} \eta)^{2}} {\bar{C}_{B}^{2}}  +\eta^{2}},
		\end{equation}
		if   $\frac{\delta}{\bar{C}_{B} \eta} = -( \frac{ \dot{v}^{\pm}_{1}}{\bar{C}_{B}}\pm 1)$, we have
		\begin{equation}\label{4.17}
			\mu^{\pm} (i\delta,\eta)= 0,
		\end{equation}
		if $ \frac{\delta}{\bar{C}_{B}\eta}< - \frac{ \dot{v}^{\pm}_{1}}{\bar{C}_{B}}-1$, we have
		\begin{equation}\label{4.18}
			\mu^{\pm} (i\delta,\eta)=-i  \sqrt{\frac{(\delta+ \dot{v}^{\pm}_{1} \eta)^{2}} {\bar{C}_{B}^{2}}  -\eta^{2}},
		\end{equation}
		if $ \frac{\delta}{\bar{C}_{B} \eta} >- \frac{ \dot{v}^{\pm}_{1}}{\bar{C}_{B}}+1$, we have
		\begin{equation}\label{4.19}
			\mu^{\pm} (i\delta,\eta)=i \sqrt{\frac{ (\delta+ \dot{v}^{\pm}_{1} \eta)^{2}} {\bar{C}_{B}^{2}}  -\eta^{2}}.
		\end{equation}
	\end{item}
\end{lemm}

\subsection{Study of  the symbol $\Sigma$}

\quad \quad From Lemmas 4.1, we can see that $\mu^{+}+ \mu^{-}=0$ in the points $(\tau,\eta)=(0, \eta)$ if and only if $\dot{v}^{+}_{1}\geq \bar{C}_{B} $. However even this situation happens, the symbol $\Sigma$ is well defined.
\begin{lemm}
	Assume $\dot{v}^{+}_{1}\geq \bar{C}_{B} $,  $\mu^{+}(0,\eta)+ \mu^{-}(0,\eta)=0$ for all $\eta\neq 0$. Then the symbol $\Sigma$ is well defined as follows:
	\begin{equation}
		\Sigma \mapsto [(\dot{v}^{+}_{1})^{2}-2\bar{C}_{B}^{2} ] \bar{\eta}^{2}, ~as~(\delta, \eta)\rightarrow (0, \bar{\eta}).
	\end{equation}
	
\end{lemm}
\begin{proof}
	First case: $\dot{v}^{+}_{1} > \bar{C}_{B}$.  Let $(\tau,\eta)= (i\delta, \eta)\in \Xi$ such that $\mathfrak{R} \tau=0, \eta\neq 0 $. From Lemma 4.1, we know that
	when $(\delta, \eta)$ in a small neighborhood of $(0,\bar{\eta})$, there holds
	$-(\frac{\dot{v}^{+}_{1} }{\bar{C}_{B}}- 1 )< \frac{\delta}{\bar{C}_{B} \eta} <\frac{\dot{v}^{+}_{1} }{\bar{C}_{B}}-1 $, therefore we have
	\begin{equation}
		\mu^{+}= i \sqrt{\frac{ (\delta+ \dot{v}^{+}_{1} \eta)^{2}} {\bar{C}_{B}^{2}}-\eta^{2}},
	\end{equation}
	and
	\begin{equation}
		\mu^{-}= -i\sqrt{\frac{ (\delta- \dot{v}^{+}_{1} \eta)^{2}} {\bar{C}_{B}^{2}}  -\eta^{2}} ,
	\end{equation}
	it follows that the sum $\mu^{+}+ \mu^{-}$ tend to zero as long as $\delta\rightarrow 0$, However, the symbol $\Sigma(\tau, \eta)$ does not tend to zero. We mainly focus the term
	$-2i \dot{v}^{+}_{1}  \eta \tau \frac{\mu^{+}-\mu^{-}}{\mu^{+}+\mu^{-}}$ with $\delta\rightarrow 0$.
	\begin{equation}\label{5.211}
		\begin{aligned}
			-2i \dot{v}^{+}_{1}  \eta \tau \frac{\mu^{+}-\mu^{-}}{\mu^{+}+\mu^{-}}
			=& 2 \dot{v}^{+}_{1}  \eta \delta  \frac{(\mu^{+} - \mu^{-})^{2}}{(\mu^{+})^{2}-(\mu^{-})^{2}}.
		\end{aligned}
	\end{equation}
	Now we compute
	\begin{equation}
		\begin{aligned}
			&(\mu^{+})^{2}-(\mu^{-})^{2}=-\frac{4 \delta \dot{v}^{+}_{1} \eta}{\bar{C}_{B}^{2}},\\
			&(\mu^{+} - \mu^{-})^{2}=-2\left[\frac{ \delta^{2}+ (\dot{v}^{+}_{1} \eta)^{2}} {\bar{C}_{B}^{2}}
			-\eta^{2} 
			+\sqrt{\frac{ (\delta+ \dot{v}^{+}_{1} \eta)^{2}} {\bar{C}_{B}^{2}}
				-\eta^{2}} \sqrt{\frac{ (\delta- \dot{v}^{+}_{1} \eta)^{2}} {\bar{C}_{B}^{2}}  -\eta^{2}}\right],
		\end{aligned}
	\end{equation}
	substituting this expression in \eqref{5.211} and  passing to the limit as $(\delta, \eta)\rightarrow (0, \bar{\eta})$ we obtain
	\begin{equation}
		\begin{aligned}
			-2i \dot{v}^{+}_{1} \eta \tau \frac{\mu^{+}-\mu^{-}}{\mu^{+}+\mu^{-}}\rightarrow  2[(\dot{v}^{+}_{1})^{2}
			-\bar{C}_{B}^{2} ] \bar{\eta}^{2}, ~as~(\delta, \eta)\rightarrow (0, \bar{\eta}).
		\end{aligned}
	\end{equation}
	
	Therefore  we have
	\begin{equation}
		\begin{aligned}
			&\Sigma \rightarrow  [(\dot{v}^{+}_{1})^{2}
			-2\bar{C}_{B}^{2} ] \bar{\eta}^{2} ~as~(\delta, \eta)\rightarrow (0, \bar{\eta}).
		\end{aligned}
	\end{equation}
	From this, we can see that the symbol $\Sigma$ is well defined.
	
	Second case: $v=\bar{C}_{B}$.  Let $(\tau,\eta)= (i\delta, \eta)\in \Xi$ such that $\mathfrak{R} \tau=0, \eta\neq 0 $. From Lemma 4.1, we know that
	when $(\delta, \eta)$ in a small neighborhood of $(0,\bar{\eta})$, if $\delta>0$,  there holds
	\begin{equation*}
		\mu^{+}= i \sqrt{\frac{ (\delta+ \dot{v}^{+}_{1} \eta)^{2}} {\bar{C}_{B}^{2}}
			-\eta^{2}} , ~ \mu^{-}= \sqrt{\frac{ -(\delta- \dot{v}^{+}_{1} \eta)^{2}} {\bar{C}_{B}^{2}} +\eta^{2} }.
	\end{equation*}
	If $\delta<0$, we have
	\begin{equation*}
		\mu^{+}=  \sqrt{\frac{ -(\delta+ \dot{v}^{+}_{1} \eta)^{2}} {\bar{C}_{B}^{2}} +\eta^{2} },
		~ \mu^{-}= - i \sqrt{\frac{ (\delta- \dot{v}^{+}_{1} \eta)^{2}} {\bar{C}_{B}^{2}}
			-\eta^{2}}.
	\end{equation*}
	From the above values of $\mu^{\pm}$ and passing to the limits as $(\delta, \eta)\rightarrow (0, \bar{\eta})$, we obtain
	\begin{equation}
		(\mu^{+}- \mu^{-})^{2} \mapsto 0 , ~as ~(\delta, \eta)\rightarrow (0, \bar{\eta}),
	\end{equation}
	hence \begin{equation}
		\begin{aligned}
			-2i \dot{v}^{+}_{1} \eta \tau \frac{\mu^{+}-\mu^{-}}{\mu^{+}+\mu^{-}}= - \frac{(\mu^{+}- \mu^{-})^{2}}{2}\rightarrow 0, ~as~(\delta, \eta)\rightarrow (0, \bar{\eta}).
		\end{aligned}
	\end{equation}
	
	Therefore  we have
	\begin{equation}
		\begin{aligned}
			&\Sigma \rightarrow -\bar{C}_{B}^{2} \bar{\eta}^{2} ~as~(\delta, \eta)\rightarrow (0, \bar{\eta}).
		\end{aligned}
	\end{equation}
	From this, we can see that the symbol $\Sigma$ is well defined.
\end{proof}

We also need to determine whether the difference $\mu^{+}-\mu^{-}$ vanishes.
\begin{lemm}
	Let $(\tau,\eta)\in \Xi$. Then $\mu^{+}=\mu^{-}$ if and only if
	\begin{item}
		\item(i) $(\tau,\eta)=(\tau,0)$.
		\item(ii) $(\tau,\eta)= (0,\eta)$, and $(\dot{v}^{\pm}_{1})^{2}<\bar{C}_{B}^{2}$.
	\end{item}
\end{lemm}

\begin{proof}
	From the equation \eqref{4.15} it follows that $(\mu^{+})^{2}=(\mu^{-})^{2}$ if and only if $\eta=0$ or $\tau=0$. If $\eta= 0$, then $\mu^{+}=\mu^{-}= \tau/\bar{C}_{B}$ which corresponds to the first case. If $\tau=0$, then $(\mu^{+})^{2}=(\mu^{-})^{2}=(1-\frac{(\dot{v}^{\pm}_{1})^{2} }{\bar{C}_{B}^{2}} ) \eta^{2}$. For $1-\frac{(\dot{v}^{\pm}_{1})^{2}}{\bar{C}_{B}^{2}}<0$  we invoke Lemma 4.1 to obtain $\mu^{\pm}=\pm i\eta \sqrt{\frac{(\dot{v}^{\pm}_{1})^{2} }{\bar{C}_{B}^{2}}-1}$,  yielding $\mu^{+}-\mu^{-}= 2 i \eta \sqrt{\frac{(\dot{v}^{\pm}_{1})^{2} }{\bar{C}_{B}^{2}-1}}\neq 0$. For $1-\frac{(\dot{v}^{\pm}_{1})^{2} }{\bar{C}_B^2}>0$ we obtain $\mu^{\pm}=
	\sqrt{1-\frac{(\dot{v}^{\pm}_{1})^{2} }{\bar{C}_{B}^{2}}}$ which corresponds to the second case.
\end{proof}

According to the definition of $\Sigma$ and Lemma 4.2,  we can easily  verify that $\Sigma(\tau,0)=\tau^{2}\neq 0$ for $(\tau,0)\in \Xi$ and $\Sigma(0,\eta)\neq 0$ for $(0,\eta)\in \Xi$. Furthermore, it can be readily checked that $\Sigma(\tau\, \eta)= \Sigma(\tau, -\eta)$. Thus, we can assume, without loss of generality, that  $\tau\neq 0$, $\eta\neq 0 $ and $ \eta>0$. From Lemma 4.3, we conclude that $\mu^{+}- \mu^{-}\neq 0$. Therefore we compute
\begin{equation}
	\begin{aligned}
		\frac{\mu^{+}-\mu^{-}}{\mu^{+}+\mu^{-}}= \frac{(\mu^{+}-\mu^{-})^{2}}{(\mu^{+})^{2}-(\mu^{-})^{2}}= \frac{\bar{C}_{B}^{2}(\mu^{+}-\mu^{-})^{2}}{4i \dot{v}^{+}_{1} \tau},
	\end{aligned}
\end{equation}
and
\begin{equation}
		(\mu^{+}-\mu^{-})^{2}= 2[(\frac{\tau}{\bar{C}_{B}})^{2}- (\frac{\dot{v}^{+}_{1}\eta}{\bar{C}_{B}})^{2} +\eta^{2}- \mu^{+}\mu^{-}],
\end{equation}
therefore we deduce that
\begin{equation}\label{4.28}
	\begin{aligned}
		\frac{\mu^{+}-\mu^{-}}{\mu^{+}+\mu^{-}}=\frac{2[{\tau}^{2}- (\dot{v}^{+}_{1}\eta)^{2} +\bar{C}_{B}^{2}(\eta^{2}- \mu^{+}\mu^{-})]}{4i \dot{v}^{+}_{1} \tau},
	\end{aligned}
\end{equation}
and substituting this last expression \eqref{4.28} into \eqref{3.155} we can rewrite it as
\begin{equation}\label{4.29}
	\begin{aligned}
		\bar{C}_{B}^{2}( \mu^{+}\mu^{-}-\eta^{2})\hat{f}+ \frac{\mu^{+} \mu^{-}}{\mu^{+}+\mu^{-}}W=0,~~if~x_{2}=0,
	\end{aligned}
\end{equation}
the symbol $\Sigma$ can be  reformulated as
\begin{equation}\label{4.30}
	\begin{aligned}
		\Sigma= \bar{C}_{B}^{2}( \mu^{+}\mu^{-}-\eta^{2}).
	\end{aligned}
\end{equation}

\section{The stability  case}
\quad \quad In this section, we will discuss the roots of  the symbol \eqref{4.30} in the stability case.
\begin{lemm}
	Let $\Sigma(\tau,\eta)$ be the symbol defined in \eqref{4.30}, for $(\tau, \eta) \in \Xi$. When $M_B>\sqrt{2}$, then $\Sigma(\tau,\eta)=0$ if and only if
		\begin{equation}\label{5.1}
			\tau= \pm i X_{1} \eta,
		\end{equation}
		where $X_{1}= \sqrt{ (\dot{v}^{+}_{1})^{2}+\bar{C}_{B}^{2}- \sqrt{\bar{C}_{B}^{4}+ 4\bar{C}_{B}^{2}(\dot{v}^{+}_{1})^{2}}}$, in this case the Kelvin-Helmholtz instability can be inhibited. Each of these root is simple.
\end{lemm}

\begin{proof}
	Let us set $\mu^{+}\mu^{-}-\eta^{2}=0$ and introduce two quantities:
	\begin{equation}\label{5.3}
		X=\frac{\tau}{i \eta},~\tilde{\mu}^{\pm}= \frac{\mu^{\pm}}{i \eta}.
	\end{equation}
	By direct calculation, we can obtain
	\begin{equation} \label{5.4}
		\tilde{\mu}^{+}\tilde{\mu}^{-}=-1,
	\end{equation}
	and
	\begin{equation}\label{5.5}
		(\tilde{\mu}^{+})^{2}(\tilde{\mu}^{-})^{2}=1.
	\end{equation}
	
	By the formula of the roots $\mu^{\pm}$, it follows that
	\begin{equation}\label{5.6}
		(\tilde{\mu}^{+})^{2}=\frac{(X+\dot{v}^{+}_{1})^{2}}{	\bar{C}_{B}^{2}}-1,
	\end{equation}
	and
	\begin{equation}\label{5.7}
		(\tilde{\mu}^{-})^{2}=\frac{(X-\dot{v}^{+}_{1})^{2}}{\bar{	C}_{B}^{2}}-1.
        \end{equation}
	Hence we have
	\begin{equation}\label{5.8}
		[(X+\dot{v}^{+}_{1})^{2}-\bar{C}_{B}^{2}][(X-\dot{v}^{+}_{1})^{2} -\bar{C}_{B}^{2}]=\bar{C}_{B}^{4},
	\end{equation}
	which leads to the following equation for $X^{2}$:
	\begin{equation}\label{5.9}
		\begin{aligned}
		X^{4} - 2 ((\dot{v}^{+}_{1})^{2}+\bar{C}_{B}^{2}) X^{2} + (\dot{v}^{+}_{1})^{4}- 2\bar{C}_{B}^{2} (\dot{v}^{+}_{1})^{2}=0.
		\end{aligned}
	\end{equation}
	Using the quadratic formula, the two roots of the above equation are
	\begin{equation}\label{5.10}
		X_{1}^{2}= (\dot{v}^{+}_{1})^{2}+\bar{C}_{B}^{2}- \sqrt{\bar{C}_{B}^{4}+ 4\bar{C}_{B}^{2}(\dot{v}^{+}_{1})^{2}},
	\end{equation}
	and
	\begin{equation}\label{5.11}
		X_{2}^{2}= (\dot{v}^{+}_{1})^{2}+\bar{C}_{B}^{2}+ \sqrt{\bar{C}_{B}^{4}+ 4\bar{C}_{B}^{2}(\dot{v}^{+}_{1})^{2}}.
	\end{equation}

	We claim that the points $(\tau,\eta)\in \Sigma$ with $\tau=\pm i X_{2} \eta$ are not the roots of $\mu^{+}\mu^{-}= \eta^{2}$. Without loss of generality, we can assume that $X_{2}$ is positive. From \eqref{5.11}, we deduce
	\begin{equation}
		X_{2}^{2}>  (\dot{v}^{+}_{1} + \bar{C}_{B})^{2},
	\end{equation}
	from this, we deduce that   $~ X_{2}> \dot{v}^{+}_{1} + \bar{ C}_{B}$,  then it follows that  $(X_{2} \pm \dot{v}^{+}_{1}) > \bar{C}_{B}^{2}$, therefore $(\tilde{\mu}^{\pm})^{2}>0$.
	By using the formula \eqref{5.6}, we have
	\begin{equation}
		\begin{aligned}
			\tilde{\mu}^{+}=  \sqrt{\frac{(X+\dot{v}^{+}_{1})^{2}}{\bar{	C}_{B}^{2}}-1}>0,
		\end{aligned}
	\end{equation}
	and
	\begin{equation}\label{5.12}
		\tilde{\mu}^{-}=\sqrt{\frac{(X-\dot{v}^{+}_{1})^{2}}{	\bar{C}_{B}^{2}}-1}>0,
	\end{equation}
	from which we know that   \eqref{5.4}  is not satisfied.
    
    Similarly, we can show that $(\tau,\eta)\in \Sigma$ with $\tau=\pm i X_{1}\eta$ is  root of $\mu^{+}\mu^{-}= \eta^{2}$.

	We aim to demonstrate that, under the condition $M_{B}>\sqrt{2}$, the roots corresponding to $X_{1}$ are simple. Since \eqref{5.4} does not admit a root at $\eta=0$, the point $(\tau,\eta)\in \Sigma$ satisfying $\mu^{\pm}=0$ are not the roots of $\mu^{+}\mu^{-}-\eta^{2}=0$. From the formula for the roots $\mu^{\pm}$, we see that $\mu^{\pm}$ are analytic in the vicinity of points where they do not vanish. Thus, we can differentiate equations \eqref{5.6} and \eqref{5.7} with respect to $X$ at $X = X_{1}$ to obtain 
	\begin{equation}\label{5.19}
		\frac{d \tilde{\mu}^{\pm}}{d X}|_{X=X_{1}} = \frac{X_{1}\pm \dot{v}^{\pm}_{1}} {\tilde{\mu}^{\pm} \bar{C}_{B}^{2}}.
	\end{equation}
	Thus
	\begin{equation}\label{5.20}
		\frac{d (\tilde{\mu}^{+}\tilde{\mu}^{-}+1)}{d X}|_{X=X_{1}} = \frac{(X_{1}+ \dot{v}^{+}_{1}) (\tilde{\mu}^{-})^{2}+(X_{1}- \dot{v}^{+}_{1}) (\tilde{\mu}^{+})^{2}} {\tilde{\mu}^{+}\tilde{\mu}^{-} \bar{C}_{B}^{2}}.
	\end{equation}
	Plugging \eqref{5.6} and \eqref{5.7} into \eqref{5.20}, we obtain
	\begin{equation}
		\begin{aligned}
			\frac{d (\tilde{\mu}^{+}\tilde{\mu}^{-}+1)}{d X}|_{X=X_{1}} 
			&=\frac{d \tilde{\mu}^{+}}{d X} \tilde{\mu}^{-}  |_{X=X_{1}} + \frac{d \tilde{\mu}^{-}}{d X} \tilde{\mu}^{+} |_{X=X_{1}}\\
			&=\frac{2X_{1}(X_{1}^{2}- (\dot{v}^{+}_{1})^{2}  -\bar{C}_{B}^{2})} {\tilde{\mu}^{+}\tilde{\mu}^{-} \bar{C}_{B}^{4}},
		\end{aligned}
	\end{equation}
	using \eqref{5.10}, we have
	\begin{equation}
		\frac{d (\tilde{\mu}^{+}\tilde{\mu}^{-}+1)}{d X}|_{X=X_{1}} \neq 0.
	\end{equation}
	Hence we have proved  $(\tau,\eta)\in \Sigma$ with $\tau=\pm i X_{1} \eta$ are all simple roots of \eqref{5.4} provided $\dot{v}^{+}_{1}>\sqrt{2} \bar{C}_{B}$. More precisely, near $\tau=\pm i X_{1} \eta$, we have $\mu^{+} \mu^{-}-\eta^{2}= (i X_{1} \eta) l^{\pm}(\tau,\eta)$ for some continuous $l^{\pm}(\tau,\eta)\neq0$ respectively.
\end{proof}

\section{Proof of Theorem 2.2}
\begin{lemm}
	Let $\Sigma$ be the symbol defined by \eqref{4.144} and $s \in \mathbb{R}$, $\gamma\geq1$. Given any $f\in H^{s+2}_{\gamma}(\mathbb{R}^{2})$, let $g$ be the function defined by
	\begin{equation}\label{6.1}
		\Sigma(\tau,\eta) \hat{f}(\tau,\eta)= \hat{g} (\tau,\eta)~~(\tau,\eta)\in \Xi,
	\end{equation}
	where $\hat{g}$ is the Fourier transform of $\tilde{g} := e^{-\gamma t}g$. Then $g\in H^{s}_{\gamma}(\mathbb{R}^{2}) $ with
	\begin{equation*}
		\|g\|_{H^{s}_{\gamma}(\mathbb{R}^{2})} \leq C \|f\|_{H^{s+2}_{\gamma}(\mathbb{R}^{2})},
	\end{equation*}
	for a suitable positive constant C independent of $\gamma$.
\end{lemm}
\begin{proof}
	From the definition of $\Sigma$, we can see that $\Sigma$ is a homogeneous of degree 2 on $\Xi$, thus there exists a positive constant $C$ such that
	\begin{equation*}
		|\Sigma(\tau, \eta)| \leq C(|\tau|^{2} + |\eta|^{2})= C \Lambda^{2}(\tau,\eta),~ for~all~(\tau, \eta)\in \Xi.
	\end{equation*}
	Then
	\begin{equation*}
		\| g\|_{H^{s}_{\gamma}(\mathbb{R}^{2})} =\frac{1}{2 \pi} \| \Lambda^{s} \hat{g} \|= \frac{1}{2 \pi} \| \Lambda^{s} \Sigma \hat{f}\| \leq C \|\Lambda^{s+2} \hat{f} \|= C \| \Lambda^{s+2} \hat{f} \|= C\| f \|_{H^{s+2}_{\gamma}(\mathbb{R}^{2})}.
	\end{equation*}
	
\end{proof}
In the following theorem we prove the a priori estimate $f$ to the equation  \eqref{6.1}, for a given g.
\begin{theo}
	Assume  $M_B>\sqrt{2}$. Let $\Sigma$ be the symbol defined by \eqref{4.144} and $s \in \mathbb{R}$, $\gamma\geq1$. Given any $f\in H^{s+2}_{\gamma}(\mathbb{R}^{2})$, let $g\in H^{s}_{\gamma}(\mathbb{R}^{2})$ be the function defined by \eqref{6.1}. Then there exists a positive constant C such that  the following estimate holds
	\begin{equation}\label{6.2}
		\gamma \|f\|_{H^{s+1}_{\gamma}(\mathbb{R}^{2})} \leq C \|g\|_{H^{s}_{\gamma}(\mathbb{R}^{2})}.
	\end{equation}
\end{theo}
\begin{proof}
	Since $\Xi$ is a $C^{\infty}$ compact manifold, there exists a finite covering $(\mathcal{V}_{1}, ..., \mathcal{V}_{I})$ of $\Xi$ by such neighborhoods, and a smooth partition of unity $(\chi_{1},. .., \chi_{I})$ associated with this covering. The $\chi^{\prime}_{i}$s are nonnegative $C^{\infty}$ functions with
	\begin{equation*}
		supp \chi_{i} \subset \mathcal{V}_{i}, ~\sum_{i=1}^{I} \chi^{2}_{i} =1.
	\end{equation*}
	We will divide two cases: in the first case $\mathcal{V}_{i}$  is a neighborhood of an elliptic point, i.e., a point $(\tau_{0}, \eta_{0})$ where $\Sigma(\tau_{0}, \eta_{0}) \neq 0$. Thus we have
	\begin{equation}\label{7.3}
		C(|\tau|^{2}+ |\eta|^{2})  |\chi_{i} \hat{f}(\tau, \eta)| \leq |\Sigma(\tau, \eta) \chi_{i} \hat{f}(\tau, \eta)| = |\chi_{i} \hat{g}(\tau, \eta)|.
	\end{equation}
	On the other hand, in the second case  $\mathcal{V}_{i}$  is a neighborhood of a root of the symbol $\Sigma$, i.e., a point $(\tau_{0}, \eta_{0})$ where $\Sigma(\tau_{0}, \eta_{0}) = 0$. Thus we have
	\begin{equation}\label{7.4}
		C\gamma (|\tau|^{2}+ |\eta|^{2})^{1/2}  |\chi_{i} \hat{f}(\tau, \eta)| \leq |\Sigma(\tau, \eta) \chi_{i} \hat{f}(\tau, \eta)| = |\chi_{i} \hat{g}(\tau, \eta)|.
	\end{equation}
	In conclusion, adding up the square of \eqref{7.3} and \eqref{7.4}    and using that the $\chi_{i}$'s form a partition of unity yields
	\begin{equation*}
		C\gamma^{2} (|\tau|^{2}+ |\eta|^{2})  |\hat{f}(\tau, \eta)|^{2}  \leq |\hat{g}(\tau, \eta)|^{2}.
	\end{equation*}
	Multiplying the previous inequality by $(|\tau|^{2}+ |\eta|^{2})^{s}$, integrating with respect to $(\delta ,\eta)\in \mathbb{R}^{2}$ and using  Plancherel's theorem finally gives
	\begin{equation*}
		C\gamma^{2} \| \tilde{f}\|^{2}_{H^{s+1}_{\gamma}(\mathbb{R}^{2})} \leq C \|g\|_{H^{s}_{\gamma}(\mathbb{R}^{2})}.
	\end{equation*}

\end{proof}

In the following theorem we prove the existence of the solution $f$ to the equation \eqref{6.1}.

\begin{theo}
	Assume  $M_B>\sqrt{2}$. Let $\Sigma$ be the symbol defined by  \eqref{4.12} and $s \in \mathbb{R}$, $\gamma\geq1$. Given any $g\in H^{s}_{\gamma}(\mathbb{R}^{2})$  there exists a unique solution $f\in H^{s+1}_{\gamma}(\mathbb{R}^{2})$ of the equation \eqref{6.1}, satisfying the estimate \eqref{6.2}.
\end{theo}
\begin{proof}
	The proof of this theorem employs a duality argument. The details of the proof are similar to those in \cite{Morando}. Thus, we will omit them here.
	
\end{proof}

Now we can complete the proof of Theorem 2.2.
\begin{proof}
	We apply the result of Theorem 6.3 for
	\begin{equation*}
		\hat{g}(\tau, \eta)= - \frac{\mu^{+} \mu^{-}}{ \mu^{+}+ \mu^{-}} W,
	\end{equation*}
	with $W$ defined in \eqref{3.16}. We write
	\begin{equation*}
		\hat{g}= \hat{g}_{1}- \hat{g}_{2},
	\end{equation*}
	where
	\begin{equation*}
		\hat{g}_{1}= - \frac{\mu^{-}}{ \mu^{+}+ \mu^{-}} \int^{+\infty}_{0} e ^{- \mu^{+} y} \hat{\mathcal{K}}^{+} dy ,~  \hat{g}_{2}=- \frac{\mu^{-}}{ \mu^{+}+ \mu^{-}} \int^{+\infty}_{0} e ^{- \mu^{+} y} \hat{\mathcal{K}}^{+} dy.
	\end{equation*}
	By the Plancherel theorem and Cauchy-Schwarz inequality we have
	\begin{equation}
		\begin{aligned}
			\| g_{1}\|^{2} _{H^{s}_{\gamma}(\mathbb{R}^{2})} &= \frac{1}{(2\pi)^{2}} \int \int_{\mathbb{R}^{2}} \Lambda^{2s}|\frac{\mu^{-}}{ \mu^{+}+ \mu^{-}}\int^{+\infty}_{0} e ^{- \mu^{+} y} \hat{\mathcal{K}}^{+} dy   |^{2} d \delta d \eta\\
			&\leq \frac{1}{(2\pi)^{2}} \int \int_{\mathbb{R}^{2}} \Lambda^{2s}|\frac{\mu^{-}}{ \mu^{+}+ \mu^{-}}|^{2}  \frac{1}{ 2 \mathfrak{R}  \mu^{+}} (\int^{+\infty}_{0}  |\hat{\mathcal{K}}^{+} |^{2} dy   )
			d \delta d \eta.
		\end{aligned}
	\end{equation}
	Then we use the fact that $\frac{\mu^{-}}{ \mu^{+}+ \mu^{-}}$ is a homogeneous of degree zero in $\Xi$ so that
	\begin{equation*}
		|\frac{\mu^{-}}{ \mu^{+}+ \mu^{-}}|^{2} \leq C,~~for ~all~(\tau, \eta) \in \Xi,
	\end{equation*}
	for a suitable constant $C>0$. Moreover, from  Lemma 5.8 in \cite{Morando},  we have the estimate from below
	\begin{equation*}
		\mathfrak{R}  \mu^{+} \geq \frac{1}{\sqrt{2} c} \gamma.
	\end{equation*}
	Thus we obtain
	\begin{equation*}
		\| g_{1}\|^{2} _{H^{s}_{\gamma}(\mathbb{R}^{2})}\leq \frac{C}{\gamma} \int \int_{\mathbb{R}^{2}} \Lambda^{2s} (\int^{+\infty}_{0}  |\hat{\mathcal{K}}^{+} |^{2} dy   )
		d \delta d \eta = \frac{C}{\gamma} \| \mathcal{K}^{+} \|^{2}_{L^{2}(\mathbb{R}^{+}; H^{s}_{\gamma}(\mathbb{R}^{2}) )}.
	\end{equation*}
	The proof of the estimate $g_{2}$ follows a similar approach. This concludes the proof of Theorem 2.2.
	
\end{proof}

\section*{Acknowledgments}
\quad \quad The work of B. Xie is supported by NSF-China under grant number 11901207, 12326355, Guangzhou Science and Technology Project 2023A04J1315 and the National Key Program  of China (2021YFA1002900).

\phantomsection

\end{document}